\documentclass[smallextended]{svjour3}       
\usepackage{amssymb}
\usepackage{graphicx}
\usepackage{hyperref}
\usepackage{verbatim,color}
\usepackage{url}
\usepackage{array}   
\newcolumntype{L}{>{$}l<{$}}
\usepackage{amssymb,amstext,amsmath}
\newcommand{\overbar}[1]{\mkern 1.5mu\overline{\mkern-1.5mu#1\mkern-1.5mu}\mkern 1.5mu}
\newcommand{\no}[1]{\overbar{#1}}
\def\F{\mathcal F}
\def\P{\mathcal P}
\def\M{\mathcal M}
\def\I{\mathcal I}
\def\H{\mathcal H}
\def\pr{\mathbb{P}}
\def\prev{\mathbb{P}}
\def\K{\mathcal{K}}
\def\G{\mathcal{G}}
\def\S{\mathcal{S}}
\def\C{\mathcal{C}}
\def\D{\mathcal{D}}
\smartqed  
\newtheorem{algorithm}{Algorithm}
\usepackage{graphicx}
\begin{document}

\title{Insert your title here
}
\title{Generalized Logical Operations among \\  Conditional Events
\footnote{This paper is a substantially extended version of \cite{GiSa17}.}}


\titlerunning{Generalized Logical Operations among   Conditional Events}        

\author{First Author         \and
        Second Author 
}


\institute{F. Author \at
              first address \\
              Tel.: +123-45-678910\\
              Fax: +123-45-678910\\
              \email{fauthor@example.com}           
           \and
           S. Author \at
              second address
}

\author{Angelo
	Gilio   \and   Giuseppe Sanfilippo
}

\institute{
	A. Gilio \at
	Department SBAI,
	University of Rome ``La Sapienza'', Italy
	\\ \email{angelo.gilio@sbai.uniroma1.it}\\
	Retired
	\and
	G. Sanfilippo \at
	Department of Mathematics and Computer Science,
	University of Palermo, Italy
	\\ \email{giuseppe.sanfilippo@unipa.it}\\
Both authors contributed equally to this work.
}

\date{Received: date / Accepted: date}

\maketitle
\begin{abstract}
We generalize, by a progressive procedure, the notions of conjunction and disjunction of two conditional events to the case of $n$ conditional events.
In our coherence-based approach, conjunctions and disjunctions are suitable conditional random quantities.
We define  the notion of negation, by verifying  De Morgan's Laws.  We also show that conjunction and disjunction satisfy the associative and commutative properties, and a monotonicity property.
Then, we give some results on coherence of prevision assessments for some families of compounded conditionals; in particular we examine the Fr\'echet-Hoeffding bounds. Moreover, we study the reverse probabilistic inference from  the conjunction $\mathcal{C}_{n+1}$ of $n+1$ conditional events  to the family $\{\mathcal{C}_{n},E_{n+1}|H_{n+1}\}$.
We  consider the relation with the notion of quasi-conjunction and 
we examine in detail the coherence of the prevision assessments related with the conjunction of  three conditional events.  
Based on conjunction, we also give a characterization of p-consistency and of p-entailment, with applications to  several inference rules in probabilistic nonmonotonic reasoning. Finally, we examine some non p-valid inference rules; then,  we illustrate by an example two methods which allow to suitably modify non p-valid inference rules in order  to get  inferences  which are p-valid.
\end{abstract}
\section{Introduction}
The research on combining logic and probability has a long
history (see, e.g., \cite{adams75,boole_1857,coletti02,defi36,hailperin96}).
In this paper we use a coherence-based approach to probability, which allows  to introduce probability assessments on arbitrary families  of conditional events, by   properly managing conditioning events of zero probability
 (see, e.g., \cite{biazzo00,biazzo05,coletti02,gilio02,gilio12ijar,gilio13,GiSa14,gilio16,SPOG18}).
In probability theory and in probability logic a relevant problem, largely discussed by many authors (see, e.g., \cite{benferhat97,CoSV13,CoSV15,GoNW91}), is that of suitably defining logical operations among conditional events. In a pioneering paper, written in 1935 by de Finetti (\cite{defi36}), it was proposed  a three-valued logic for conditional events coinciding  with that one of Lukasiewicz.  
A survey of the many  contributions by different authors (such as Adams, Belnap, Calabrese, de Finetti, Dubois, van Fraassen, McGee, Goodmann, Lewis, Nguyen, Prade, Schay) 
to research on  three-valued logics and compounds of conditionals has been given in \cite{Miln97};
conditionals have  also been extensively studied in \cite{edgington95,McGe89}.
In the literature, the conjunction and disjunction  have been usually defined as  suitable conditionals; see e.g.   
\cite{adams75,Cala87,CiDu12,GoNW91}.
A theory for the compounds of conditionals has been proposed in  \cite{McGe89,Kauf09}.  A related  theory  has been  developed in the setting of coherence in \cite{GiSa13c,GiSa13a,GiSa14}; in these papers,  conjunction and disjunction of two conditional events are not defined as conditional events,
 but as suitable {\em conditional random quantities}, with values in the interval $[0,1]$. 
In the present paper we generalize the notions of conjunction and disjunction of two conditional events to the case of $n$ conditional events; we also give the notion of negation.
Then, we examine a monotonicity property for conjunction and disjunction. Moreover,  we give some results on coherence of prevision assessments for some families of compounded conditionals; in particular we examine the  Fr\'echet-Hoeffding bounds. Finally, we examine in detail the coherence of prevision assessments related with the conjunction of three conditional events. 
The paper is organized as described below. In Section~\ref{SECT:PRELIMINARIES} we 
recall some preliminary notions and results which concern coherence, quasi conjunction, conjunction, disjunction, and   Fr\'echet-Hoeffding bounds.
In Section \ref{SECT:CONJUNCTIONn} we introduce, in a progressive way, the notions of conjunction and disjunction for $n$ conditional events; then, we define the notion of  negation  and  we show that De Morgan's Laws are satisfied.  We  define the notion of conjunction (resp., disjunction) for the conjunctions (resp., disjunctions) associated with two families of conditional events, by showing then the validity of commutative and associative properties.   In Section~\ref{SECT:MONOTONICITY}, after  a preliminary result concerning  the inequality $X|H\leq Y|K$ between two conditional random quantities, we show that the conjunction $\C_{n+1}$  of $n+1$ conditional events  is a conditional random quantity less than or equal to  any conjunction  $\C_n$ of a subfamily of $n$  conditional events. Likewise, we show that the disjunction $\D_{n+1}$  of $n+1$ conditional events  is greater than or equal to  any disjunction  $\D_n$ of a subfamily of $n$  conditional events. We also show that $\C_n$ and $\D_n$ belong to  the interval $[0,1]$. Moreover, we derive some inequalities from the monotony property. In Section~\ref{SECT:TETRAEDRO}, based on a geometrical approach, we characterize by an iterative procedure the set of coherent assessments on the family $\{C_n,E_{n+1}|H_{n+1},\C_{n+1}\}$. In Section \ref{SECT:BACKPROPAGATION}  we
study the (reverse) inference from $\C_{n+1}$
to $\{\C_{n},E_{n+1}|H_{n+1}\}$, by determining the
set of coherent extensions $(\mu_n,x_{n+1})$ of any coherent assessment $\mu_{n+1}$,
 where $\mu_{n}=\prev(\C_{n})$, $x_{n+1}=P(E_{n+1}|H_{n+1})$, and $\mu_{n+1}=\prev(\C_{n+1})$.
In Section \ref{SECT:FRECHET} we show that  the prevision of the conjunction $\C_n$ satisfies the  Fr\'echet-Hoeffding bounds. Then, by exploiting De Morgan's Laws, we give the dual result for the disjunction $\D_n$.
In Section \ref{SECT:CONJUNCTION3} we examine in detail the conjunction for a family of three conditional events $E_1|H_1,E_2|H_2,E_3|H_3$.  We also consider the relation with the notion  of  quasi-conjunction studied in \cite{adams75}; see also \cite{GiSa13IJAR,gilio13}.  We also determine  the set of coherent prevision assessments on the whole family   $\{E_1|H_1,E_2|H_2,E_3|H_3,
(E_1|H_1)\wedge (E_2|H_2), (E_1|H_1)\wedge (E_3|H_3), (E_2|H_2)\wedge (E_3|H_3),(E_1|H_1)\wedge (E_2|H_2)\wedge (E_3|H_3)\}$. Moreover, we consider the particular case where $H_1=H_2=H_3=H$.
In Section \ref{SECT:APPL}, by applying our notion  of conjunction, we give a  characterization of p-consistency and p-entailment and we examine  some p-valid inference rules in probabilistic nonmonotonic reasoning.  Moreover, based on a suitable notion of iterated conditioning, we briefly describe a characterization of p-entailment in the case of two premises.
In Section \ref{SECT:NPV-PV}, after examining some non p-valid inference rules, 
we  illustrate two methods which allow to construct p-valid inferences. 
Finally, in Section \ref{SECT:CONCLUSIONS} we give a summary of  results.
Notice that for almost all (new) results of this paper the proofs  are given in Appendix \ref{SECT:APPENDIXA}.
\section{Some Preliminaries} \label{SECT:PRELIMINARIES}
In this section we recall some basic notions and results on coherence (see, e.g., \cite{biazzo05,BiGS12,CaLS07,coletti02,PeVa17}).
In our approach an event $A$ represents an uncertain fact described by a (non ambiguous) logical proposition; hence  $A$ is a two-valued logical entity which can be \emph{true}, or \emph{false}.
The indicator of $A$, denoted by the same symbol, is  1, or 0, according to  whether $A$ is true, or false. The sure event is denoted by $\Omega$ and the impossible event is denoted by $\emptyset$.  Moreover, we denote by $A\land B$, or simply $AB$, (resp., $A \vee B$) the logical conjunction (resp., logical disjunction).   The  negation of $A$ is denoted $\no{A}$. Given any events $A$ and $B$, we simply write $A \subseteq B$ to denote that $A$ logically implies $B$, that is  $A\no{B}$ is the impossible event $\emptyset$. We recall that  $n$ events are logically independent when the number $m$ of constituents, or possible worlds, generated by them  is $2^n$ (in general  $m\leq 2^n$).
\subsection{Conditional events and coherent  probability assessments}
\label{sect:2.2}
Given two events $E,H$,
with $H \neq \emptyset$, the conditional event $E|H$
is defined as a three-valued logical entity which is \emph{true}, or
\emph{false}, or \emph{void}, according to whether $EH$ is true, or $\no{E}H$
is true, or $\no{H}$ is true, respectively.
We recall that, agreeing to
the betting metaphor, if you assess $P(E|H)=p$, then, for every  real number $s$, you are willing
to pay  an amount $ps$ and  to receive $s$, or 0, or $ps$, according to
whether $EH$ is true, or $\no{E}H$ is true, or $\no{H}$ is true (bet
called off), respectively. 
Then, the random gain associated with the assessment $P(E|H)=p$ is 
$G = sHE+ps\no{H}-ps=sH(E-p)$. Given a real function $P : \; \mathcal{K}
\, \rightarrow \, \mathcal{R}$, where $\mathcal{K}$ is an arbitrary
family of conditional events, let us consider a subfamily 
$\mathcal{F} = \{E_1|H_1, \ldots, E_n|H_n\}$ of $\mathcal{K}$
 and the vector $\mathcal{P} =(p_1, \ldots, p_n)$,
where $p_i = P(E_i|H_i) \, ,\;\; i \in J_n=\{1,\ldots,n\}$. We denote by
$\mathcal{H}_n$ the disjunction $H_1 \vee \cdots \vee H_n$. As
$E_iH_i \vee \no{E}_iH_i \vee \no{H}_i = \Omega \,,\;\; i \in J_n$,
by expanding the expression $\bigwedge_{i\in J_n}(E_iH_i \vee \no{E}_iH_i
\vee \no{H}_i)$ we can represent $\Omega$ as the disjunction of $3^n$
logical conjunctions, some of which may be impossible. 
The
remaining ones are the constituents generated by 
$\mathcal{F}$ and, of course, are a partition of $\Omega$.
We denote by $C_1, \ldots, C_m$ the constituents which logically imply 
$\mathcal{H}_n$ and (if $\mathcal{H}_n \neq \Omega$) by $C_0$ the
remaining constituent $\mathcal{\no{H}}_n = \no{H}_1 \cdots \no{H}_n$, so
that
\[
\mathcal{H}_n = C_1 \vee \cdots \vee C_m \,,\;\;\; \Omega =
\mathcal{\no{H}}_n \vee
\mathcal{H}_n = C_0 \vee C_1 \vee \cdots \vee C_m \,,\;\;\; m+1 \leq 3^n
\,.
\]
In the context of betting scheme, with the pair $(\mathcal{F}, \mathcal{P}$) we associate the random gain $G = \sum_{i\in J_n} s_iH_i(E_i - p_i)$,
where $s_1, \ldots, s_n$ are $n$ arbitrary real numbers. We observe that $G$ is the difference between the amount that you receive, $\sum_{i\in J_n} s_i(E_iH_i + p_i\no{H}_i)$, and the amount that you pay, $\sum_{i\in J_n} s_ip_i$, and represents the net gain from engaging each transaction $H_i(E_i - p_i)$, the scaling and meaning (buy or sell) of the transaction being specified by the magnitude and the sign of $s_i$, respectively.
Let $g_h$
be the value of $G$ when $C_h$ is true; then $G\in \{g_0,g_1,\ldots,g_m\}$.
Of course, $g_0 = 0$. We denote by  $\mathcal{G}_{\mathcal{H}_n}$  the set of values of $G$ restricted to $\H_n$, that is $\mathcal{G}_{\mathcal{H}_n}=\{g_1, \ldots, g_m\}$. 
Then, based on the  {\em betting  scheme} of de Finetti, we
have
\begin{definition}\label{COER-BET}  The function $P$ defined on $\mathcal{K}$ is said to be {\em coherent}
	if and only if, for every integer $n$, for every finite subfamily $\mathcal{F}$ of 
	$\mathcal{K}$ and for every  real numbers $s_1, \ldots, s_n$, one has:
	$\min  \mathcal{G}_{\mathcal{H}_n} \leq 0 \leq \max \mathcal{G}_{\mathcal{H}_n}$.
\end{definition}
Notice that the condition $\min \mathcal{G}_{\mathcal{H}_n}  \leq 0 \leq \max \mathcal{G}_{\mathcal{H}_n}$ can be written in two equivalent ways: $\min  \mathcal{G}_{\mathcal{H}_n} \leq 0$, or  $\max  \mathcal{G}_{\mathcal{H}_n} \geq 0$.  As shown by Definition \ref{COER-BET}, a probability assessment is coherent if and only if, in any finite combination of $n$ bets, it does not happen that the values $g_1, \ldots, g_m$ are all positive, or all negative ({\em no Dutch Book}).

\subsection{Coherent conditional prevision assessments}
\label{Coherence}
 Given a prevision function $\pr$ defined on an arbitrary family $\K$ of finite
conditional random quantities, consider a finite subfamily $\F = \{X_1|H_1, \ldots,X_n|H_n\} \subseteq \K$ and the vector
$\M=(\mu_1,\ldots, \mu_n)$, where $\mu_i = \pr(X_i|H_i)$ is the
assessed prevision for the conditional random quantity $X_i|H_i$, $i\in J_n$.
With the pair $(\F,\M)$ we associate the random gain $G =
\sum_{i \in J_n}s_iH_i(X_i - \mu_i)$; moreover, we denote by $\G_{\mathcal{H}_n}$ the set of values of $G$ restricted to $\H_n= H_1 \vee \cdots \vee H_n$. Then, by the {\em  betting scheme}, we
have
\begin{definition}\label{COER-RQ}{\rm
		The function $\pr$ defined on $\K$ is coherent if and only if, $\forall n
		\geq 1$,  $\forall \, \F \subseteq \K,\, \forall \, s_1, \ldots,
		s_n \in \mathbb{R}$, it holds that: $min \; \G_{\mathcal{H}_n} \; \leq 0 \leq max \;
		\G_{\mathcal{H}_n}$. }
\end{definition}
Given a family $\F = \{X_1|H_1,\ldots,X_n|H_n\}$, for each $i \in J_n$ we denote by $\{x_{i1}, \ldots,x_{ir_i}\}$ the set of possible values for the restriction of $X_i$ to $H_i$; then, for each $i \in J_n$ and $j = 1, \ldots, r_i$, we set $A_{ij} = (X_i = x_{ij})$. Of course, for each $i \in J_n$, the family $\{\no{H}_i, A_{ij}H_i \,,\; j = 1, \ldots, r_i\}$ is a partition of the sure event $\Omega$, with  $A_{ij}H_i=A_{ij}$, $\bigvee_{j=1}^{r_i}A_{ij}=H_i$. Then,
the constituents generated by the family $\F$ are (the
elements of the partition of $\Omega$) obtained by expanding the
expression $\bigwedge_{i \in J_n}(A_{i1} \vee \cdots \vee A_{ir_i} \vee
\no{H}_i)$. We set $C_0 = \no{H}_1 \cdots \no{H}_n$ (it may be $C_0 = \emptyset$);
moreover, we denote by $C_1, \ldots, C_m$ the constituents
contained in $\H_n$. Hence
$\bigwedge_{i \in J_n}(A_{i1} \vee \cdots \vee A_{ir_i} \vee
\no{H}_i) = \bigvee_{h = 0}^m C_h$.
With each $C_h,\, h \in J_ m$, we associate a vector
$Q_h=(q_{h1},\ldots,q_{hn})$, where $q_{hi}=x_{ij}$ if $C_h \subseteq
A_{ij},\, j=1,\ldots,r_i$, while $q_{hi}=\mu_i$ if $C_h \subseteq \no{H}_i$;
with $C_0$ it is associated  $Q_0=\M = (\mu_1,\ldots,\mu_n)$.
Denoting by $\I$ the convex hull of $Q_1, \ldots, Q_m$, the condition  $\M\in \I$ amounts to the existence of a vector $(\lambda_1,\ldots,\lambda_m)$ such that:
$ \sum_{h \in J_m} \lambda_h Q_h = \M \,,\; \sum_{h \in J_m} \lambda_h
= 1 \,,\; \lambda_h \geq 0 \,,\; \forall \, h$; in other words, $\M\in \I$ is equivalent to the solvability of the system $(\Sigma)$, associated with  $(\F,\M)$,
\begin{equation}\label{SYST-SIGMA}
\begin{array}{l}
(\Sigma) \;\;\;\sum_{h \in J_m} \lambda_h q_{hi} =
\mu_i \,,\; i \in J_n \,; \; \sum_{h \in J_m} \lambda_h = 1 \,;\;
\lambda_h \geq 0 \,,\;  \, h\in J_m \,.
\end{array}
\end{equation}
Given the assessment $\M =(\mu_1,\ldots,\mu_n)$ on  $\F =
\{X_1|H_1,\ldots,X_n|H_n\}$, let $S$ be the set of solutions $\Lambda = (\lambda_1, \ldots,\lambda_m)$ of system $(\Sigma)$ defined in  (\ref{SYST-SIGMA}).   
Then, the  following characterization theorem for coherent assessments on  finite families of conditional events can be proved (\cite{BiGS08})
\begin{theorem}\label{SYSTEM-SOLV}{ \rm [{\em Characterization of coherence}].
		Given a family of $n$ conditional random quantities $\F =
		\{X_1|H_1,\ldots,X_n|H_n\}$, with finite sets of possible values, and a vector $\M =
		(\mu_1,\ldots,\mu_n)$, the conditional prevision assessment
		$\prev(X_1|H_1) = \mu_1 \,,\, \ldots \,,\, \prev(X_n|H_n) =
		\mu_n$ is coherent if and only if, for every subset $J \subseteq J_n$,
		defining $\F_J = \{X_i|H_i \,,\, i \in J\}$, $\M_J = (\mu_i \,,\,
		i \in J)$, the system $(\Sigma_J)$ associated with the pair
		$(\F_J,\M_J)$ is solvable. }
	\end{theorem}
We point out that the solvability of  system $(\Sigma)$ (i.e., the condition $\M\in \I$) is a necessary (but not sufficient) condition for coherence of $\M$ on $\F$. Moreover, assuming  the system $(\Sigma)$  solvable, that is  $S \neq \emptyset$, we define:
\begin{equation}\label{EQ:I0}
\begin{array}{ll}
I_0 = \{i : \max_{\Lambda \in S} \, \sum_{h:C_h\subseteq H_i}\lambda_h= 0\},\;  \F_0 = \{X_i|H_i \,, i \in I_0\},\;  \M_0 = (\mu_i ,\, i \in I_0)\,.
\end{array}
\end{equation}
Then, the following theorem can be proved  (\cite[Theorem 3]{BiGS08})
\begin{theorem}\label{CNES-PREV-I_0-INT}{\rm [{\em Operative characterization of coherence}]
		A conditional prevision assessment ${\M} = (\mu_1,\ldots,\mu_n)$ on
		the family $\F = \{X_1|H_1,\ldots,X_n|H_n\}$ is coherent if
		and only if the following conditions are satisfied: \\
		(i) the system $(\Sigma)$ defined in (\ref{SYST-SIGMA}) is solvable;\\ (ii) if $I_0 \neq \emptyset$, then $\M_0$ is coherent. }
\end{theorem}
 By Theorem \ref{CNES-PREV-I_0-INT}, the
following algorithm checks in a finite number of steps the coherence of the prevision 
 assessment $\M$ on $\F$.
\begin{algorithm}\label{ALG-PREV-INT}{\rm
		Let be given the pair $(\F, \M)$.
		\begin{enumerate}
			\item Construct the system $(\Sigma)$ defined in (\ref{SYST-SIGMA}) and check its
			solvability; 
			\item If the  system $(\Sigma)$  is not
			solvable then $\M$ is not coherent and the procedure stops,
			otherwise compute the set $I_0$; 
			\item If $I_0 = \emptyset$ then
			$\M$ is coherent and the procedure stops, otherwise set
			$(\F, \M) = (\F_0, M_0)$ and repeat steps
			1-3.
		\end{enumerate}
}\end{algorithm}  
By following the approach given in \cite{CoSc99,GOPS16,GiSa13c,GiSa13a,GiSa14,lad96}  a conditional random quantity $X|H$ can be seen as the random quantity $XH+\mu\no{H}$, where $\mu=\prev(X|H)$.
In particular, in numerical terms, $A|H$ is the random quantity $AH+x\no{H}$, where $x=P(A|H)$. 
Then, when $H\subseteq A$, coherence requires that $P(A|H)=1$ and hence $A|H=H+\no{H}=1$. 
Notice that, as  $\no{H}|H=0$ and $XH|H=X|H$, it holds that: $(XH+\mu\no{H})|H=XH|H+\mu\no{H}|H=X|H$, where $\mu=\prev(X|H)$.
Moreover, the negation of $A|H$ is defined as  $\no{A|H}=1-A|H=\no{A}|H$. Coherence can be characterized in terms of proper scoring rules (\cite{BiGS12,GiSa11a}), which can be related  to the notion of entropy in information theory (\cite{LSA12,LSA15}).

We recall a result  (see \cite[Theorem 4]{GiSa14})  which  shows that, given two conditional random quantities $X|H$, $Y|K$, if  $X|H= Y|K$ when $H\vee K$ is true, then $X|H= Y|K$ also when  $H\vee K$ is false, so that $X|H=Y|K$.
\begin{theorem}\label{THM:EQ-CRQ}{\rm Given any events $H\neq \emptyset$, $K\neq \emptyset$, and any r.q.'s $X$, $Y$, let $\Pi$ be the set of the coherent prevision assessments $\pr(X|H)=\mu,\pr(Y|K)=\nu$. \\
		(i) Assume that, for every $(\mu,\nu)\in \Pi$,  $X|H=Y|K$  when  $H\vee K$ is true; then   $\mu= \nu$ for every $(\mu,\nu)\in \Pi$. \\
		(ii) For every $(\mu,\nu)\in \Pi$,   $X|H=Y|K$  when  $H\vee K$ is true  if and only if $X|H= Y|K$.
}\end{theorem}

\subsection{Quasi conjunction, conjunction, and disjunction of two conditional events.} 
The notion of quasi conjunction plays an important role in nonmonotonic reasoning. In particular for two conditional events $A|H, B|K$ the quasi conjunction $QC(A|H,B|K)$  is the  conditional event $(\no{H}\vee A)\wedge(\no{K}\vee B)\,|\, (H \vee K)$.
Note that:  $QC(A|H,B|K)$ is true, when a conditional event is true and the other one is not false; $QC(A|H,B|K)$ is false, when a conditional event is false; $QC(A|H,B|K)$ is void, when $H\vee K$ is false. In other words,  the quasi conjunction is the conjunction of the two material conditionals $\no{H}\vee A, \no{K}\vee B$ given the disjunction of the conditioning events $H,K$. 
In numerical terms one has
\begin{equation}\label{eq:QC}
QC(A|H,B|K) =\min \, \{\no{H}\vee A, \no{K}\vee B\} \,|\, (H \vee K)
\end{equation}
and, if we replace the material conditionals $\no{H}\vee A, \no{K}\vee B$ by
the conditional events $A|H,B|K$, from formula (\ref{eq:QC}) we obtain the definition below (\cite{GiSa13a}).
\begin{definition}\label{CONJUNCTION}{\rm Given any pair of conditional events $A|H$ and $B|K$, with $P(A|H)=x$, $P(B|K)=y$, we define their conjunction
		as the conditional random quantity $(A|H) \wedge (B|K) = Z\,|\, (H \vee K)$, where $Z=\min \, \{A|H, B|K\}$.
}\end{definition}
Then, defining $z=\mathbb{P}[(A|H)\wedge(B|K)]$, we have
\begin{equation}\label{EQ:CONJUNCTION}
(A|H)\wedge(B|K) =\left\{\begin{array}{ll}
1, &\mbox{ if $AHBK$  is true,}\\
0, &\mbox{ if  $\no{A}H\vee \no{B}K$ is true,}\\
x, &\mbox{ if $\no{H}BK$ is true,}\\
y, &\mbox{ if $AH\no{K}$ is true,}\\
z, &\mbox{ if $\no{H}\no{K}$ is true}.
\end{array}
\right.
\end{equation}
\begin{remark}
We recall that  $A|H=AH+x\no{H}$, where $x=P(A|H)$. Then, by Definition \ref{CONJUNCTION}, it holds that  
\[
(A|H) \wedge (A|H)=(A|H)|H= (AH+x\no{H})|H= AH|H=A|H.
\]
\end{remark}
From (\ref{EQ:CONJUNCTION}), the conjunction  $(A|H) \wedge (B|K)$ is the following  random quantity
\begin{equation}\label{EQ:REPRES}
(A|H) \wedge (B|K)=1 \cdot AHBK + x \cdot \no{H}BK + y \cdot AH\no{K} + z \cdot \no{H}\no{K}\,.
\end{equation}
For the quasi conjunction it holds that
\begin{equation}
\label{EQ:QCREPRES}
QC(A|H,B|K)= AHBK +  \no{H}BK +   AH\no{K} + \nu  \cdot \no{H}\no{K},
\end{equation}
where $\nu=P(QC(A|H,B|K))$. We recall that, if $P(A|H)=P(B|K)=1$, then  $\nu=1$ (see, e.g.,  \cite[Section 3]{gilio13}).
We also recall a result  which shows that Fr\'echet-Hoeffding bounds still hold for the conjunction of conditional events (\cite[Theorem~7]{GiSa14}).
\begin{theorem}\label{THM:FRECHET}{\rm
		Given any coherent assessment $(x,y)$ on $\{A|H, B|K\}$, with $A,H,B$, $K$ logically independent, $H\neq \emptyset, K\neq  \emptyset$, the extension $z = \mathbb{P}[(A|H) \wedge (B|K)]$ is coherent if and only if the following  Fr\'echet-Hoeffding bounds are satisfied:  $\max\{x+y-1,0\} = z' \; \leq \; z \; \leq \; z'' =\min\{x,y\}$.
}\end{theorem}
\begin{remark}\label{REM:QCAND}
We observe that,  if $x=y=1$, then coherence requires that $z=\nu=1$ and then by  (\ref{EQ:REPRES}) and (\ref{EQ:QCREPRES}) it follows that
$(A|H) \wedge (B|K)=QC(A|H,B|K)$.
\end{remark}
We  recall now the notion of disjunction of two conditional events.
\begin{definition}\label{DEF:DISJUNCTION}{\rm Given any pair of conditional events $A|H$ and $B|K$, with $P(A|H)=x$, $P(B|K)=y$, we define their disjunction
		as  $(A|H) \vee (B|K) = W\,|\, (H \vee K)$, where $W=\max \, \{A|H, B|K\}$.
}\end{definition}
Then, defining $w=\mathbb{P}[(A|H)\vee (B|K)]$, we have
\begin{equation}\label{EQ:DISJUNCTION}
(A|H)\vee (B|K) =\left\{\begin{array}{ll}
1, &\mbox{ if  $AH\vee BK$ is true,}\\
0, &\mbox{ if $\no{A}H\no{B}K$  is true,}\\
x, &\mbox{ if $\no{H}\no{B}K$ is true,}\\
y, &\mbox{ if $\no{A}H\no{K}$ is true,}\\
w, &\mbox{ if $\no{H}\no{K}$ is true}.
\end{array}
\right.
\end{equation}
\begin{remark}
	We recall that  $A|H=AH+x\no{H}$, where $x=P(A|H)$. Then, by Definition \ref{DEF:DISJUNCTION}, it holds that  
	\[
	(A|H) \vee (A|H)=(A|H)|H= (AH+x\no{H})|H= AH|H=A|H.
	\]
\end{remark}
From (\ref{EQ:DISJUNCTION}), the disjunction $(A|H) \vee (B|K)$ is the following  random quantity
\begin{equation}
(A|H) \vee (B|K)=1 \cdot AH\vee BK + x \cdot \no{H}\no{B}K + y \cdot \no{A}H\no{K} + w \cdot \no{H}\no{K}.
\end{equation}

\section{Conjunction,  Disjunction,  and  Negation}
\label{SECT:CONJUNCTIONn}	
We now define the conjunction and the disjunction of $n$ conditional events in a progressive  way by specifying  the possible values of the corresponding conditional random quantities. Given a family of $n$ conditional events $\mathcal{F}=\{E_1|H_1,\ldots,E_n|H_n\}$, we denote by $C_0,C_1,\ldots, C_m$, with $m+1\leq 3^n$, the constituents associated with $\mathcal{F}$, where $C_0=\no{H}_1\no{H}_2\cdots \no{H}_n$.
With each $C_h$, $h=0,1,\ldots,m$,
we  associate a tripartition $(S_h',S_h'',S_h''')$ of the set $\{1,\ldots,n\}$, such that, for each $i\in\{1,\ldots,n\}$ it holds that:  $i\in S_h'$, or $i\in S_h''$, or  $i\in S_h'''$, according to whether  $C_h\subseteq E_iH_i$,  or $C_h\subseteq \no{E}_iH_i$, or $C_h\subseteq \no{H}_i$.
In other words, for each $h=0,1,\ldots,m$,  we have
\begin{equation}\label{EQ:ESSE}
\begin{array}{lll}
S_h'=\{i: C_h\subseteq E_iH_i\}, &
S_h''=\{i: C_h\subseteq \no{E}_iH_i\},  &
S_h'''=\{i: C_h\subseteq \no{H}_i\}\,.
\end{array}
\end{equation}
\begin{definition}[Conjunction of $n$ conditionals]\label{DEF:CONJUNCTIONn}
	{\rm
		Let be given a family of $n$ conditional events $\mathcal{F}=\{E_1|H_1,\ldots,E_n|H_n\}$.
		For each  non-empty subset $S$  of $\{1,\ldots,n\}$,  let $x_{S}$ be a  prevision assessment on $\bigwedge_{i\in S} (E_i|H_i)$.
		Then, the conjunction $\C(\F)=(E_1|H_1) \wedge \cdots \wedge (E_n|H_n)$ is defined as
		\begin{equation}\label{EQ:CF}
		\begin{array}{lll}
		Z_n|(H_1\vee \cdots \vee H_n)=\sum_{h=0}^{m} z_{h}C_h, & \mbox{where} &
		z_{h}=
		\left\{
		\begin{array}{llll}
		1, &\mbox{ if } S_h'=\{1,\ldots,n\},\\
		0, &\mbox{ if } S_h''\neq \emptyset, \\
		x_{S'''_h}, &\mbox{ if } S_h''=\emptyset \mbox{ and }  S_h''' \neq\emptyset\,. \\
		\end{array}
		\right.
		\end{array}
		\end{equation}
}\end{definition}
\begin{remark}\label{REM:CONJUNCTIONn}
As shown by   (\ref{EQ:CF}), the conjunction $(E_1|H_1) \wedge \cdots \wedge (E_n|H_n)$ assumes one of the following possible values:
1, when every  conditional event is true;
0, when at least one  conditional event is false;
$x_S$, when the conditional event $E_i|H_i$  is void, for every  $i\in S$, and  is true for every  $i\notin S$.
In the case $S=\{i\}$, we simply set $x_S=x_i$.
 \end{remark}
 Notice that  the notion of conjunction given in (\ref{EQ:CF}) has been already proposed, with positive probabilities for the conditioning events, in \cite{McGe89}. But,  our approach is developed in the setting of  coherence, where  conditional probabilities and conditional  previsions are primitive notions. Moreover, coherence allows to properly manage zero probabilities for conditioning events.
\begin{remark}\label{REM:PERMUTATION}
	We observe that to introduce the random quantity defined by
	formula (\ref{EQ:CF})  we need to specify in a coherent way the set of prevision assessments $\{x_S: S\subseteq \{1,2,\ldots,n\}\}$.
	In particular, when the conditioning events $H_1,\ldots, H_n$ are all false, i.e. $C_0$ is true, the associated tripartition is $(S_0',S_0'',S_0''')=(\emptyset,\emptyset,\{1,2,\ldots,n\}$) and the value of the conjunction $\C(\F)$ is its prevision
	$x_{S_0'''}=\pr[\C(\F)] $. Moreover,
	we observe that the set of the constituents $\{C_0,\ldots,C_m\}$ associated with $\F$ is invariant with respect to any permutation of the conditional events in $\F$.
	Then, the operation of conjunction introduced by  Definition \ref{DEF:CONJUNCTIONn} is
	invariant with respect to any permutation of the conditional events in $\F$.
\end{remark}
\begin{definition}\label{DEF:CONJUNCTIONFAM}
	Given two finite families of conditional events  $\F'$ and $\F''$, based on Definition \ref{DEF:CONJUNCTIONn}, we set $\C(\F')\wedge \C(\F'')=\C(\F'\cup \F'')$.
\end{definition}

\begin{proposition}\label{PROP:ASS}
The operation of conjunction is associative and commutative.
\end{proposition}
\begin{proof}
	Concerning the commutative property, 	let be given two finite families of conditional events  $\F'$ and $\F''$.
	As $\F'' \cup \F'=\F' \cup \F''$, it holds that $\C(\F'')\wedge \C(\F')=\C(\F''\cup \F')=\C(\F'\cup \F'')=\C(\F')\wedge \C(\F'')$.
	Concerning the associative property, 	let be given three finite families of conditional events  $\F', \F''$ and $\F'''$.	
We have 
\[
\begin{array}{ll}
[\C(\F')\wedge \C(\F'')]\wedge  \C(\F''')=\C(\F'\cup \F'')\wedge  \C(\F''')=\C(\F'\cup \F''\cup \F''')=\\
=
 \C(\F')\wedge \C(\F''\cup \F''')=\C(\F')\wedge [\C(\F'')\wedge  \C(\F''')]=\C(\F')\wedge \C(\F'')\wedge  \C(\F''').
 \end{array}
\]	\qed
\end{proof}

\begin{definition}[Disjunction of $n$ conditionals]\label{DEF:DISJUNCTIONn}
	{\rm
		Let be given a 	family of $n$ conditional events $\mathcal{F}=\{E_1|H_1,\ldots,E_n|H_n\}$. Morever, for each  non-empty subset $S$  of $\{1,\ldots,n\}$, let $y_{S}$ be a  prevision assessment on $\bigvee_{i\in S} (E_i|H_i)$.\\
		Then, the disjunction $\mathcal{D}(\F)=(E_1|H_1) \vee \cdots \vee (E_n|H_n)$ is defined as the following conditional random quantity
		\begin{equation}\label{EQ:DF}
		\begin{array}{lll}
		W_n|(H_1\vee \cdots \vee H_n)=\sum_{h=0}^{m} w_{h}C_h,
		&\mbox{where} &
		w_{h}=
		\left\{
		\begin{array}{llll}
		1, &\mbox{ if } S_h'\neq \emptyset,\\
		0, &\mbox{ if } S_h''=\{1,2,\ldots,n\}, \\
		y_{S'''_h}, &\mbox{ if } S_h'=\emptyset \mbox{ and }  S_h''' \neq\emptyset\,. \\
		\end{array}
		\right.
		\end{array}
		\end{equation}
}\end{definition}
We recall that  $S_0'''=\{1,2,\ldots,n\}$; thus $y_{S_0'''}=\pr[\bigvee_{i=1}^n (E_i|H_i )]=\prev[\D(\F)]$.
As shown by   (\ref{EQ:DF}), the disjunction $\mathcal{D}(\F)$ assumes one of the following possible values: 1, when at least one conditional event is true;
0, when every  conditional event is false;
$y_S$, when the conditional event $E_i|H_i$  is void, for every  $i\in S$, and  is false for every  $i\notin S$.

\begin{definition}\label{DEF:DISJUNCTIONFAM}
	Given two finite families of conditional events  $\F'$ and $\F''$, based on Definition \ref{DEF:DISJUNCTIONn}, we set $\D(\F')\vee \D(\F'')=\D(\F'\cup \F'')$.
\end{definition}
\begin{proposition}\label{PROP:DISJASS}
	The operation of disjunction is associative and commutative.
\end{proposition}
\begin{proof}
The proof is analogous to that of Proposition \ref{PROP:ASS}.	
\qed
\end{proof}
We give below the notion of negation for the conjunction and the disjunction of a family of conditional events.
\begin{definition}\label{DEF:NEGATION}
	Given a family  of conditional events $\mathcal{F}$,
	the negations for the conjunction $\C(\F)$ and the  disjunction $\D(\F)$ are defined as
	$\no{\C(\F)}=1-\C(\F)$ and	$\no{\D(\F)}=1-\D(\F)$, respectively.
\end{definition}
Given a family of $n$ conditional events $\mathcal{F}=\{E_1|H_1,\ldots,E_n|H_n\}$, we denote by $\no{\F}$ the family $\{\no{E}_1|H_1,\ldots,\no{E}_n|H_n\}$. Of course $\no{\no{\F}}=\F$.
In the next result we show that  De Morgan's Laws are satisfied.
\begin{theorem}\label{THM:DEMORGAN}
Given a family of $n$ conditional events $\mathcal{F}=\{E_1|H_1,\ldots,E_n|H_n\}$, it holds that: \\
(i) $\no{\D(\F)}=\C(\no{\F})$, that is $\D(\F)=\no{\C(\no{\F})}$; \\
(ii) $\no{\C(\F)}=\D(\no{\F})$, that is $\C(\F)=\no{\D(\no{\F})}$. \\
\end{theorem}
\begin{proof}
	See Appendix \ref{SECT:APPENDIXA}.
\end{proof}
\section{Monotonicity property}
\label{SECT:MONOTONICITY}
For any given $n$ conditional events $E_1|H_1,\ldots,E_n|H_n$, we set $\mathcal{C}_n=\bigwedge_{i=1}^n(E_i|H_i)$ and
$\mathcal{D}_n=\bigvee_{i=1}^n(E_i|H_i)$. Moreover, for every non empty subset $S$ of $\{1,2,\ldots,n\}$ we set
\[
\C_{S}=\bigwedge_{i\in S}(E_i|H_i),\;\; \D_{S}=\bigvee_{i\in S}(E_i|H_i)\,.
\]
In this section, among other results, we will  show the monotonicity property of  conjunction and disjunction, that is
 $\C_{n+1}\leq \C_{n}$ and $\D_{n+1}\geq \D_{n}$, for every $n\geq 1$.\\
We first prove a preliminary result, which in particular shows that, given two conditional random quantities $X|H$, $Y|K$, if  $X|H\leq Y|K$ when $H\vee K$ is true, then $X|H\leq Y|K$ also when  $H\vee K$ is false, so that $X|H\leq Y|K$.
This result  generalizes Theorem \ref{THM:EQ-CRQ}, as the symbol $=$ is replaced by  $\leq$, and it will be used in Theorem~\ref{THM:MONOTONY}.
\begin{theorem}\label{THM:INEQ-CRQ}{\rm Given any events $H\neq \emptyset$, $K\neq \emptyset$, and any r.q.'s $X$, $Y$, let $\Pi$ be the set of the coherent prevision assessments $\pr(X|H)=\mu,\pr(Y|K)=\nu$. \\
		(i) Assume that, for every $(\mu,\nu)\in \Pi$,  $X|H\leq Y|K$  when  $H\vee K$ is true; then   $\mu\leq \nu$ for every $(\mu,\nu)\in \Pi$. \\
		(ii) For every $(\mu,\nu)\in \Pi$,   $X|H\leq Y|K$  when  $H\vee K$ is true  if and only if $X|H\leq Y|K$.
}\end{theorem}
\begin{proof}
	See Appendix \ref{SECT:APPENDIXA}.
\end{proof}
The next two results illustrate the monotonicity property of  conjunction and disjunction.
\begin{theorem}\label{THM:MONOTONY}
Given  $n+1$ arbitrary conditional events $E_1|H_1,\ldots,E_{n+1}|H_{n+1}$, with $n\geq 1$, for the conjunctions $\mathcal{C}_n$ and $\mathcal{C}_{n+1}$ it holds that  $\mathcal{C}_{n+1} \leq \mathcal{C}_{n}$.
\end{theorem}
\begin{proof}
	See Appendix \ref{SECT:APPENDIXA}.
\end{proof}
\begin{theorem}\label{THM:MONOTONYDISJ}
Given  $n+1$ arbitrary conditional events $E_1|H_1,\ldots,E_{n+1}|H_{n+1}$, with $n\geq 1$, for the disjunctions $\mathcal{D}_n$ and $\mathcal{D}_{n+1}$ it holds that  $\mathcal{D}_{n+1} \geq \mathcal{D}_{n}$.
\end{theorem}
\begin{proof}
	Defining  $\F_{n}=\{E_1|H_1,\ldots,E_{n}|H_{n}\}$ and $\F_{n+1}=\F_n\cup \{E_{n+1}|H_{n+1}\}$, 
by   Theorems \ref{THM:DEMORGAN} and \ref{THM:MONOTONY} it holds that
	\[	
	\begin{array}{l}
	\mathcal{D}_{n+1}=	\mathcal{D}(\F_{n+1})=\no{\mathcal{C}(\no{\F}_{n+1} )}=	1-\mathcal{C}(\no{\F}_{n+1} )\geq 1-\mathcal{C}(\no{\F}_{n} )=\no{\mathcal{C}(\no{\F}_{n} )}=\mathcal{D}_{n}.
	\end{array}
	\]
	\qed
\end{proof}
The next result shows that the conjunction and the disjunction are random quantities with values in the interval $[0,1]$.
\begin{theorem}	\label{THM:INTCn}
Given  $n$ arbitrary conditional events $E_1|H_1,\ldots,E_{n}|H_{n}$, it holds that: $(i)$ $\C_n\in [0,1]$; $(ii)$ $\D_n\in [0,1]$.
\end{theorem}	
\begin{proof}
	See Appendix \ref{SECT:APPENDIXA}.
\end{proof}
\begin{remark}\label{REM:MONOTONY}
	From Theorem~\ref{THM:MONOTONY}, it holds that $\C_{n}\leq \C_{n-1} \leq \ldots \leq \C_1$; in particular
	$\prev(\C_{n})\leq \prev(\C_{k})$, $k=1,2,\ldots,n-1$.
	More  generally, for every non empty subset $S$ of $\{1,\ldots,n\}$, it holds that
	$\prev(\C_{n})\leq \prev(\C_S)$. In particular, $\prev(\C_{n})\leq P(E_{n}|H_{n})$. Then,
	$\prev(\C_{k})\leq \min\{\prev(\C_{k-1}),P(E_{k}|H_{k})\}$,  $k=2,3,\ldots, n $,
	and by iterating it follows
	\begin{equation}\label{EQ:MIN}
	\prev(\C_{n})\leq \min\{P(E_1|H_1), \ldots, P(E_{n}|H_{n})\}.
	\end{equation}
	Likewise, by Theorem \ref{THM:MONOTONYDISJ},	$\prev(\D_{k})\geq \max\{\prev(\D_{k-1}),P(E_{k}|H_{k})\}$,  $k=2,3,\ldots, n$, and by iterating it follows
	\begin{equation}\label{EQ:MAX}
	\prev(\D_{n})\geq \max\{P(E_1|H_1), \ldots, P(E_{n}|H_{n})\}.
	\end{equation}
\end{remark}	
\section{Coherent assessments on   $\{\C_n, E_{n+1}|H_{n+1}, \C_{n+1}\}$}
\label{SECT:TETRAEDRO}
Given any $n+1$ arbitrary conditional events $E_1|H_1,\ldots, E_{n+1}|H_{n+1}$, let us consider the conjunctions $\C_n=(E_1|H_1)\wedge\cdots \wedge (E_{n}|H_{n})$
and 
$\C_{n+1}=(E_1|H_1)\wedge\cdots \wedge (E_{n+1}|H_{n+1})$.
We set $\prev(\C_n)=\mu_n$, $\prev(\C_{n+1})=\mu_{n+1}$ and $P(E_{n+1}|H_{n+1})=x_{n+1}$.
\begin{remark}\label{REM:FRECH}
	Let us consider the points  
	\[
	Q_1=(1,1,1),\; Q_2=(1,0,0),\; Q_3=(0,1,0),\; Q_4=(0,0,0).
	\]
	We observe that the equations of the three planes containing the points $Q_1,Q_2,Q_3$, or $Q_1,Q_2,Q_4$, or $Q_1,Q_3,Q_4$, are $z=x+y-1$, or $z=x$, or $z=y$, respectively. It can be shown that a point $(x,y,z)$ belongs to the convex hull $\I$ of $Q_1,Q_2,Q_3,Q_4$  if and only if 
	\begin{equation}\label{EQ:LOWUP-FRECHET}
	(x,y) \in [0,1]^2 \,,\;\;\; \max\{x+y-1,0\} \leq z \leq \min\{x,y\} \,.
	\end{equation}
	The convex hull $\I$, which is a tetrahedron  with vertices $Q_1,Q_2,Q_3,Q_4$, is depicted in Figure \ref{FIG:IEA1}.
\end{remark}
We observe that the lower and upper bounds in (\ref{EQ:LOWUP-FRECHET})  are	the Fr\'echet-Hoeffding bounds, which characterize the next result.
\begin{figure}[tbph]
\centering
\includegraphics[width=0.7\linewidth]{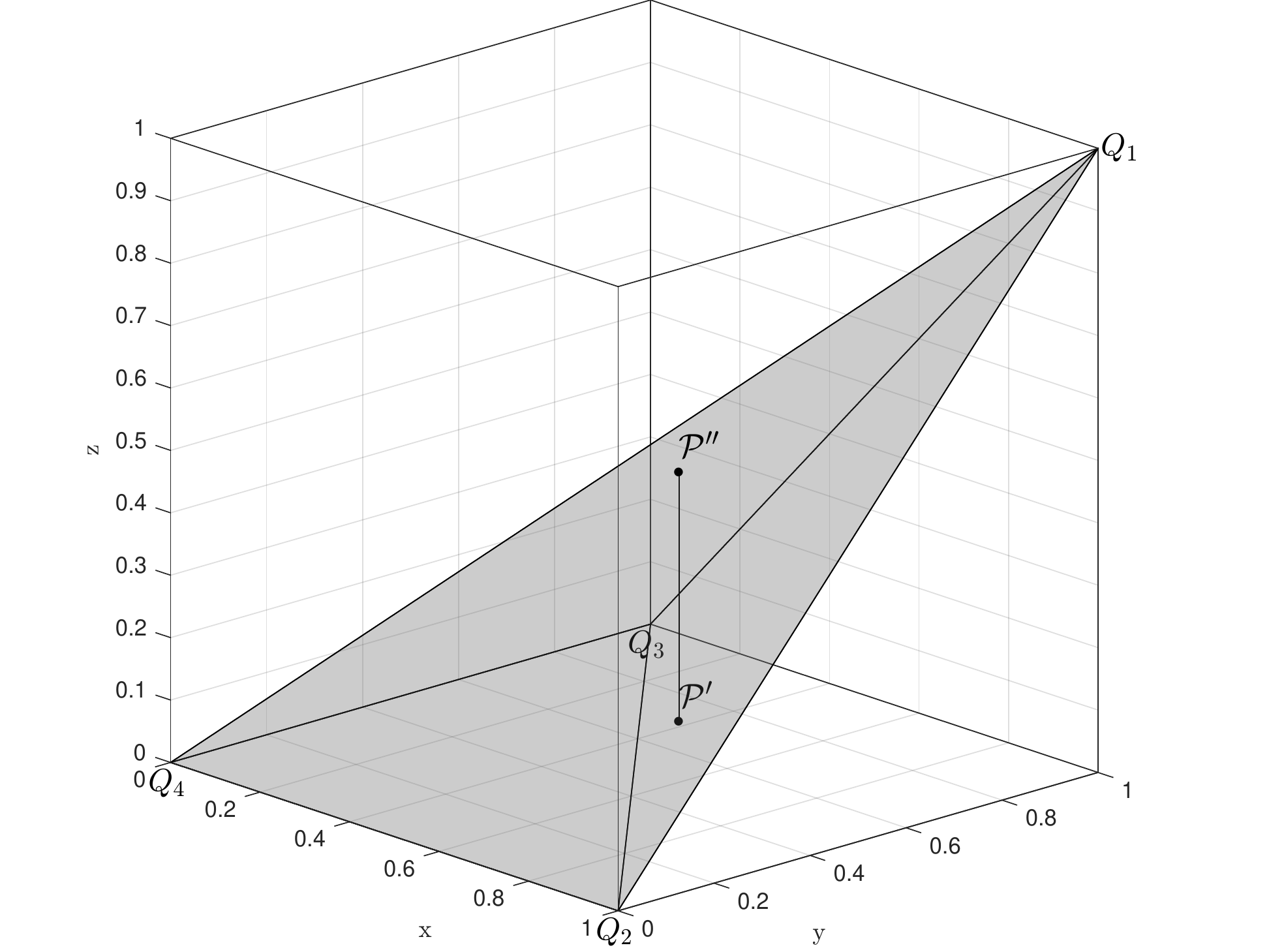}
\caption{Convex hull of  the points $Q_1, Q_2,Q_3, Q_4$.  
$\P'=(x,y,z'), \P''=(x,y,z'')$, where $(x,y)\in[0,1]^2$, $z'=\max\{x+y-1,0\}$, $z''=\min\{x,y\}$. In the figure the numerical  values are: $x=0.6$, $y=0.5$, $z'=0.1$, and  $z''=0.5$.}
\label{FIG:IEA1}
\end{figure}

\begin{theorem}\label{THM:FRECHETCn}
	Assume that the events $E_{1},\dots,E_{n+1},H_{1},\dots,H_{n+1}$ are logically independent.
	Let $\I$ be the convex hull  of the points $Q_1=(1,1,1),\; Q_2=(1,0,0),\; Q_3=(0,1,0),\; Q_4=(0,0,0)$.
	Then, the  assessment $\M=(\mu_n,x_{n+1},\mu_{n+1})$ on the family
	$\{\C_n, E_{n+1}|H_{n+1}, \C_{n+1}\}$ is coherent if and only if $\M\in \I$, that is if and only if
	\[
	(\mu_n,x_{n+1})\in[0,1]^2, \;\;\;\mu_{n+1}'
	\leq \mu_{n+1}\leq \mu_{n+1}'',
	\]
	where  $\mu_{n+1}'=\max\{\mu_n+x_{n+1}-1,0\}$ and $\mu_{n+1}''=
	\min\{\mu_n,x_{n+1}\}$.
\end{theorem}
\begin{proof}
	See Appendix \ref{SECT:APPENDIXA}.
\end{proof}
\begin{remark}  We observe that the representation of each coherent assessment $\M=(\mu_n,x_{n+1},\mu_{n+1})$ as a linear convex combination  $\lambda_1Q_1 +\lambda_2Q_2+\lambda_3Q_3+\lambda_4Q_4$  (where $\sum_{h=1}^4{\lambda_h}=1$, $\lambda_h\geq 0$, $h=1,2,3,4$ ) is unique, with 
\[
\left\{
\begin{array}{ll}
\lambda_1=\mu_{n+1}=\prev(\C_{n+1}) \geq 0,\\
\lambda_2=\mu_n-\mu_{n+1}=\prev(\C_n) - \prev(\C_{n+1} )\geq 0,\\
\lambda_3=x_{n+1}-\mu_{n+1}=P(E_{n+1}|H_{n+1})-\prev(\C_{n+1})  \geq 0,\\
\lambda_4=1-\mu_n-x_{n+1}+\mu_{n+1} = 1 - \prev(\C_n) - P(E_{n+1}|H_{n+1})+\prev(\C_{n+1} )\geq 0 \,.
\end{array}
\right.
\]	
In particular, concerning the extreme cases $\mu_{n+1}=\mu_{n+1}'$, or $\mu_{n+1}=\mu_{n+1}''$, 
we can examine four cases: 1) $\mu_{n+1}'=\mu_n+x_{n+1} - 1 > 0$; 2) $\mu_{n+1}'=0$; \linebreak 3) $\mu_{n+1}''=\mu_n$ and 4)  $\mu_{n+1}''=x_{n+1}$. \\ In the case 1 the point  $\M=(\mu_n,x_{n+1},\mu_{n+1})$ is a linear convex combination  $\lambda_1Q_1 +\lambda_2Q_2+\lambda_3Q_3+\lambda_4Q_4$, with
$\lambda_1=\mu'_{n+1}=\mu_n+x_{n+1}-1, \lambda_2=1-x_{n+1}, \lambda_3=1-\mu_n,\lambda_4=0
$.\\
	In the case 2 it holds that
$
	\lambda_1=\mu'_{n+1}=0,
	\lambda_2=\mu_n,
	\lambda_3=x_{n+1},
	\lambda_4=1-\mu_n+x_{n+1}
$.\\
	In the case 3 it holds that
$
	\lambda_1=\mu''_{n+1}=\mu_n \,,\;
	\lambda_2=0 \,,\;
	\lambda_3=x_{n+1}-\mu_n \,,\;
	\lambda_4=1-x_{n+1}
$.\\
	In the case 4 it holds that
$
	\lambda_1=\mu''_{n+1}=x_{n+1} \,,\;
	\lambda_2=\mu_n-x_{n+1} \,,\;
	\lambda_3=0 \,,\;
	\lambda_4=1-\mu_n
$.
\end{remark}	
\section{Probabilistic Inference from $\C_{n+1}$ to $\{\C_{n},E_{n+1}|H_{n+1}\}$}
\label{SECT:BACKPROPAGATION}
In this section, given any coherent prevision assessment $\mu_{n+1}$ on $\C_{n+1}$, we find the set of coherent  extensions  $ (\mu_{n},x_{n+1})$ on $\{\C_{n}, E_{n+1}|H_{n+1}\}$.
As we will see, it is enough to  illustrate the case $n=1$, by finding the
set of coherent extensions $(x,y)$  on $\{E_1|H_1,E_2|H_2\}$ of any  assessment $z=\pr[ (E_1|H_1)\wedge (E_2|H_2)]\in[0,1]$.
\begin{theorem}\label{THM:INVFRECHET}{\rm
		Given any  prevision assessment $z$ on $(E_1|H_1)\wedge (E_2|H_2)$, with $z\in[0,1]$, with $E_1,H_1,E_2,H_2$ logically independent, with $H_1\neq \emptyset$ and $H_2\neq \emptyset$, the extension ${x= P(E_1|H_1)}$, $y=P(E_2|H_2)$ is coherent if and only if $(x,y)$ belongs to the set $T_z=\{(x,y):x\in[z,1], y\in[z,1+z-x]\}$.
}\end{theorem}
\begin{proof}
	We recall  that, by logical independence of $E_1,H_1,E_2,H_2$, the assessment $(x,y)$ is coherent for every  $(x,y)\in [0,1]^2$. From Theorem \ref{THM:FRECHET}, the set $\Pi$ of all coherent assessment $(x,y,z)$ on $\{E_1|H_1,E_2|H_2,(E_1|H_1)\wedge (E_2|H_2)\}$ is
	$\Pi=\{(x,y,z): (x,y)\in[0,1]^2, \max\{x+y-1,0\}  \leq \; z \; \leq \; \min\{x,y\}\}$.
	We note that
	\[
	\begin{array}{lll}
	\Pi&=&\{(x,y,z): z\in[0,1],  x \in [z,1], y\in [z,z+1-x]\}=\\
	&=&\{(x,y,z): z\in[0,1],(x,y)\in T_z\}.
	\end{array}	
	\]
	Then,  $(x,y)$ is a coherent extension of $z$ if and only if
	$(x,y)\in T_z$.
	\qed
\end{proof}
\begin{figure}[tbph]
	\centering
	\includegraphics[width=0.7\linewidth]{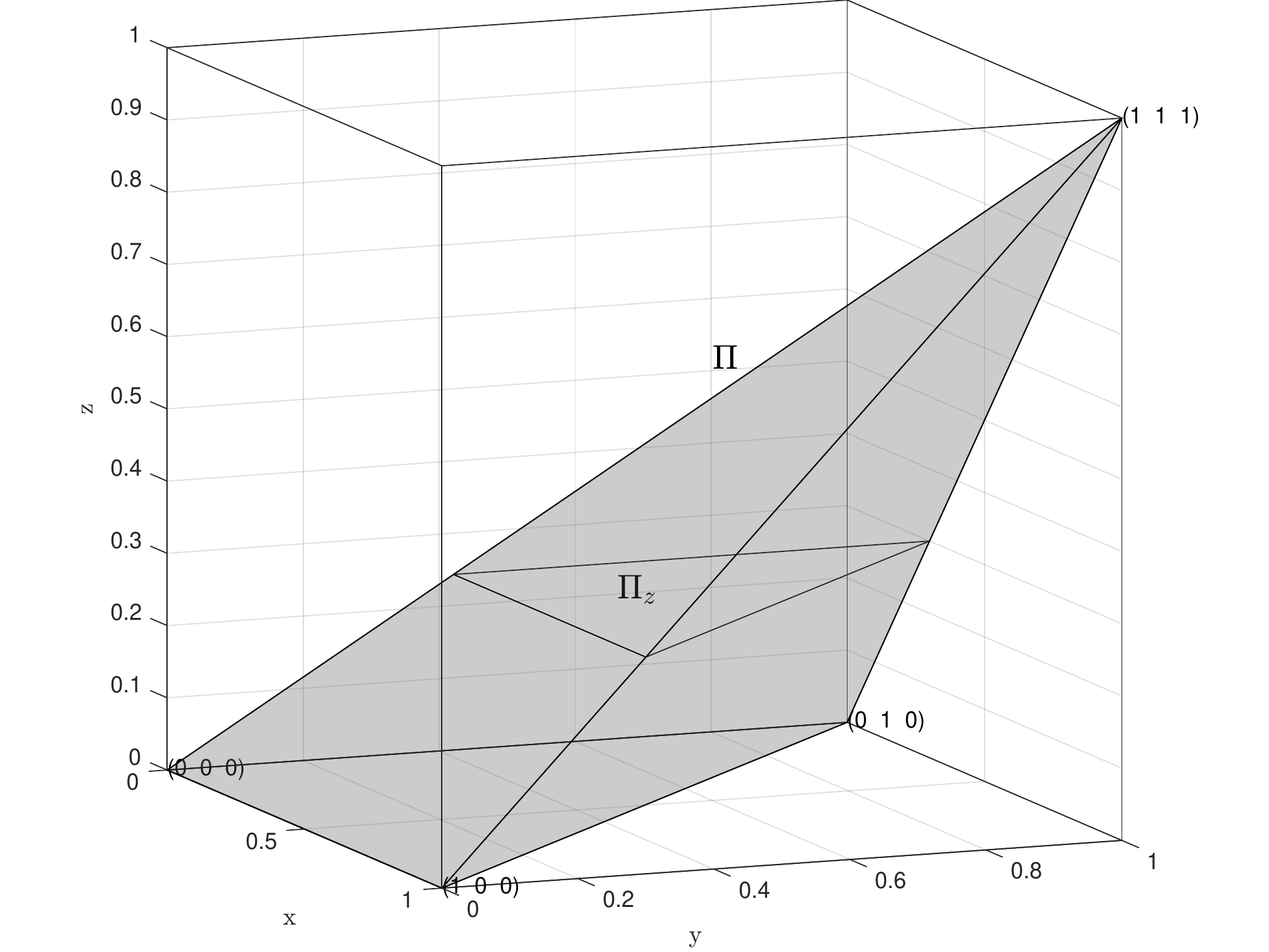}
	\caption{Set $\Pi$ of all coherent  assessments $(x,y,z)$ on $\{E_1|H_1,E_2|H_2,(E_1|H_1)\wedge (E_2|H_2)\}$. Notice that  $\Pi=\bigcup_{z\in[0,1]} \Pi_{z}$, where for each given  $z\in[0,1]$  the set $\Pi_{z}$ is the triangle $\{(x,y,z):(x,y)\in T_z\}$, with
		$T_z=\{(x,y):x\in[z,1], y\in[z,1+z-x]\}$.}
	\label{FIG:IEA2}
\end{figure}
\begin{remark}
	We observe that, given any $z\in[0,1]$ and defining $\Pi_z=\{(x,y,z): (x,y)\in T_z \}$, it holds that $\Pi=\bigcup_{z\in[0,1]}\Pi_z$ (see Figure \ref{FIG:IEA2}). The set $\Pi$ is the tetrahedron depicted in Figure \ref{FIG:IEA1}.
	Hence, contrarily to the general case, for the family  $\{E_1|H_1,E_2|H_2,(E_1|H_1)\wedge (E_2|H_2)\}$ the set of coherent prevision assessments $\Pi$ is convex. Indeed, $\Pi$ is also  the (convex) set of coherent probability assessment $(x,y,z)$ on the family of unconditional events $\{E_1,E_2,E_1E_2\}$. We recall that, assuming $H_1\wedge H_2=\emptyset$, the set of coherent prevision assessments $(x,y,z)$ on $\{E_1|H_1,E_2|H_2,(E_1|H_1)\wedge (E_2|H_2)\}$ is the surface  $\{(x,y,z): (x,y)\in[0,1]^2,  z =xy\}$,
	which is a strict non-convex subset of $\Pi$ (see \cite[Section 5]{GiSa13a}).
\end{remark}
\begin{theorem}\label{THM:INVFRECHETn}{\rm
		Given any  prevision assessment $\mu_{n+1}=\prev(\C_{n+1})\in[0,1]$, with $\mu_{n+1}\in[0,1]$, 
		the extension ${\mu_n= \prev(\C_n)}$, $x_{n+1}=P(E_{n+1}|H_{n+1})$ is coherent if and only if 
		\[
		(\mu_n,x_{n+1})\in \{(\mu_n,x_{n+1}):\mu_n\in[\mu_{n+1},1], x_{n+1}\in[\mu_{n+1},1+\mu_{n+1}-\mu_n]\}.
		\]
}\end{theorem}
\begin{proof}
From Theorem \ref{THM:FRECHETCn}, the set $\Pi$ of all coherent assessment $(\mu_n,x_{n+1},\mu_{n+1})$ on $\{\C_n,E_{n+1}|H_{n+1},\C_{n+1}\}$ is
	$\Pi=\{(\mu_n,x_{n+1},\mu_{n+1}): (\mu_n,x_{n+1})\in[0,1]^2, \max\{\mu_{n}+x_{n+1}-1,0\}  \leq \; \mu_{n+1} \; \leq \; \min\{\mu_n,x_{n+1}\}\}$. Moreover, as observed in the proof of Theorem  \ref{THM:INVFRECHET}, the set $\Pi$ coincides with the set
	\[
	\{(\mu_n,x_{n+1},\mu_{n+1}): \mu_{n+1}\in[0,1],  \mu_n\in[\mu_{n+1},1], x_{n+1}\in[\mu_{n+1},1+\mu_{n+1}-\mu_n]\}.
	\]
Then,  $(\mu_n,x_{n+1})$ is a coherent extension of $\mu_{n+1}$ if and only if $(\mu_n,x_{n+1})$ belongs to the set $\{(\mu_n,x_{n+1}):\mu_n\in[\mu_{n+1},1], x_{n+1}\in[\mu_{n+1},1+\mu_{n+1}-\mu_n]\}$.
	\qed
\end{proof}
\section{Fr\'echet-Hoeffding Bounds}
\label{SECT:FRECHET}
In the next result we show that the prevision of the conjunction $\C_n=E_1|H_1\wedge \cdots \wedge E_n|H_n  $ satisfies the  Fr\'echet-Hoeffding bounds.
\begin{theorem}\label{THM:LUKMIN}
Let be given $n$ conditional events $E_1|H_1, E_2|H_2, \ldots, E_n|H_n$, with $x_i=P(E_i|H_i)$, $i=1,2,\ldots, n$, and with $\prev(\C_n)=\mu_n$. Then 
\begin{equation}\label{EQ:LUKMIN}
\max\{x_{1}+\cdots +x_{n}-(n-1),0\}\; \leq\; \mu_n \;\leq\; \min\{x_1,\ldots,x_n\}\,.
\end{equation}
\end{theorem}	
\begin{proof}
 From Theorem~\ref{THM:FRECHETCn}, it holds that
\[
\mu_n \geq \mu_{n-1}+x_{n}-1 \geq \mu_{n-2}+x_{n-1}+x_{n}-2 \geq \cdots \geq x_{1}+\cdots +x_{n}-(n-1).
\]	
Then, by inequality	(\ref{EQ:MIN}) and by Theorem~\ref{THM:INTCn} it holds that 
the inequalities in (\ref{EQ:LUKMIN}) are satisfied.
\qed
\end{proof}
Likewise, the following result holds for the prevision $\eta_n$ of  the disjunction $\D_n=E_1|H_1\vee E_2|H_2 \vee \cdots\vee E_n|H_n$. 
\begin{theorem}
	Let be given $n$ conditional events $E_1|H_1, E_2|H_2, \ldots, E_n|H_n$, with $x_i=P(E_i|H_i)$, $i=1,2,\ldots, n$, and with $\prev(\D_n)=\eta_n$. Then 
	\begin{equation}\label{EQ:MAXMIN}
	\max\{x_1,\ldots,x_n\}\; \leq\; \eta_n \;\leq\; \min\{ x_{1}+\cdots +x_{n},1\} \,.
	\end{equation}
\end{theorem}	
\begin{proof}
By  Definition~\ref{DEF:NEGATION},
Theorems \ref{THM:DEMORGAN} and \ref{THM:LUKMIN}, defining $\no{\F}_{n}=\{\no{E}_1|H_1, \no{E}_2|H_2, \ldots, \no{E}_n|H_n\}$ it holds that 
	\[
	\begin{array}{ll}
	\prev(\D_n)=1-\prev(\no{\D}_n)=1-\prev(\C(\no{\F}_{n}) )=1-\prev(\bigwedge_{i=1}^n (\no{E_i}|H_i) )\;\;\leq \\ \leq\;\;
1-[(1-x_1)+\cdots +(1-x_n)-(n-1)] =
	x_1+\cdots+x_n .
	\end{array}
	\]
	Then, by (\ref{EQ:MAX}) and by Theorem~\ref{THM:INTCn}, 
the inequalities in (\ref{EQ:MAXMIN}) are satisfied.
\qed	
\end{proof}
\section{Conjunction of Three Conditional Events}\label{SECT:CONJUNCTION3}
Given a family of three conditional events $\F=\{E_1|H_1, E_2|H_2, E_3|H_3$\}, we set $P(E_i|H_i) = x_i$, $i=1,2,3$, $\mathbb{P}[(E_i|H_i)\wedge(E_j|H_j)]=x_{ij}=x_{ji}$, $i\neq j$, and  $x_{123}=\mathbb{P}[(E_1|H_1) \wedge (E_2|H_2)\wedge (E_3|H_3)]$.	
Then, by Definition~\ref{DEF:CONJUNCTIONn}, the conjunction of $E_1|H_1, E_2|H_2, E_3|H_3$ is the conditional random quantity 
\begin{equation}\label{EQ:CONJUNCTION3}
\small
\C(\F)=	(E_1|H_1) \wedge (E_2|H_2)\wedge (E_3|H_3)=
\left\{
\begin{array}{llll}
1, &\mbox{ if } E_1H_1E_2H_2E_3H_3 \mbox{ is true}\\
0, &\mbox{ if } \no{E}_1H_1 \vee \no{E}_2H_2 \vee \no{E}_3H_3 \mbox{ is true},\\
x_1,& \mbox{ if } \no{H}_1E_2H_2E_3H_3 \mbox{ is true},\\
x_2,& \mbox{ if } \no{H}_2E_1H_1E_3H_3 \mbox{ is true},\\
x_3, &\mbox{ if } \no{H}_3E_1H_1E_2H_2 \mbox{ is true}, \\
x_{12}, &\mbox{ if } \no{H}_1\no{H}_2E_3H_3 \mbox{ is true}, \\
x_{13}, &\mbox{ if } \no{H}_1\no{H}_3E_2H_2 \mbox{ is true}, \\
x_{23}, &\mbox{ if } \no{H}_2\no{H}_3E_1H_1 \mbox{ is true}, \\
x_{123}, &\mbox{ if } \no{H}_1\no{H}_2\no{H}_3 \mbox{ is true}. \\
\end{array}
\right.
\end{equation}
\begin{remark}
	Notice that in the betting scheme $x_{123}$ is the quantity to be paid in order to receive $\C(\F)$.
	Assuming that  the assessment $(x_1,x_2,x_3,x_{12},x_{13},x_{23})$ on $\{E_1|H_1,E_2|H_2,E_3|H_3,(E_1|H_1)\wedge(E_2|H_2),(E_1|H_1)\wedge(E_3|H_3),(E_2|H_2)\wedge(E_3|H_3)\}$ is coherent, we are interested in finding the values  $x_{123}$ which are a coherent extension of $(x_1,x_2,x_3,x_{12},x_{13},x_{23})$. Of course,  as $x_i\in[0,1]$, $i=1,2,3$,  and $x_{ij}\in [0,1]$, $i\neq j$,   a necessary condition for coherence  is   $x_{123}\in [0,1]$.
\end{remark}
From  Remark \ref{REM:PERMUTATION} and  Proposition \ref{PROP:ASS}   the conjunction $\C(\F)$ is invariant with  respect to 	any given permutation $(i_1,i_2,i_3)$ of $(1,2,3)$; that is  $\C(\F)=(E_{i_1}|H_{i_1}) \wedge (E_{i_2}|H_{i_2}))\wedge (E_{i_3}|H_{i_3})$, for any permutation $(i_1,i_2,i_3)$ of $(1,2,3)$.

\subsection{Study of Coherence}
Notice that in general, if  we  assess  the values $x_{S}=\prev(\C_{S})$ for some $S\subset \{1,2\ldots,n\}$, then the study of coherence may be very complex. In this section  we study coherence in the case $n=3$ when we assess the prevision $x_{S}=\prev(\C_{S})$ for every $S\subseteq\{1,2,3\}$.
In the next result we determine the set  of  coherent assessments $\mathcal{M}=(x_1,x_2,x_3,x_{12},x_{13},x_{23},x_{123})$ on the family
$\F=\{E_1|H_1,E_2|H_2,E_3|H_3, (E_1|H_1)\wedge (E_2|H_2),$ $(E_1|H_1)\wedge (E_3|H_3),
(E_2|H_2)\wedge (E_3|H_3)$, ${(E_1|H_1)\wedge (E_2|H_2)\wedge (E_3|H_3)}\}=\{\C_{S}: \emptyset \neq S \subseteq \{1,2,3\}\}$.

\label{SECT:PI}
\begin{theorem}\label{THM:PIFOR3}
Assume that  the events $E_1, E_2, E_3, H_1, H_2, H_3$  are logically independent, with $H_1\neq \emptyset, H_2\neq \emptyset, H_3\neq \emptyset$.
Then,	the set $\Pi$ of all coherent assessments  $\mathcal{M}=(x_1,x_2,x_3,x_{12},x_{13},x_{23},x_{123})$ on
$\F=\{E_1|H_1,E_2|H_2,E_3|H_3, (E_1|H_1)\wedge (E_2|H_2),$ $(E_1|H_1)\wedge (E_3|H_3),
(E_2|H_2)\wedge (E_3|H_3)$, ${(E_1|H_1)\wedge (E_2|H_2)\wedge (E_3|H_3)}\}$ is the set of points $(x_1,x_2,x_3,x_{12},x_{13},x_{23},x_{123})$ 
which satisfy the following  conditions
\begin{equation}
\label{EQ:SYSTEMPISTATEMENT}
\left\{
\begin{array}{l}
(x_1,x_2,x_3)\in[0,1]^3,\\
\max\{x_1+x_2-1,x_{13}+x_{23}-x_3,0\}\leq x_{12}\leq \min\{x_1,x_2\},\\
\max\{x_1+x_3-1,x_{12}+x_{23}-x_2,0\}\leq x_{13}\leq \min\{x_1,x_3\},\\
\max\{x_2+x_3-1,x_{12}+x_{13}-x_1,0\}\leq x_{23}\leq \min\{x_2,x_3\},\\
1-x_1-x_2-x_3+x_{12}+x_{13}+x_{23}\geq 0,\\
x_{123}\geq \max\{0,x_{12}+x_{13}-x_1,x_{12}+x_{23}-x_2,x_{13}+x_{23}-x_3\},\\
x_{123}\leq  \min\{x_{12},x_{13},x_{23},1-x_1-x_2-x_3+x_{12}+x_{13}+x_{23}\}.
\end{array}
\right.
\end{equation}
\end{theorem}
\begin{proof}
	See Appendix \ref{SECT:APPENDIXA}.
\end{proof}	
We observe that, from (\ref{EQ:SYSTEMPISTATEMENT}) it  follows that the  coherence of $(x_1,x_2,x_3,x_{12},x_{13},x_{23})$ amounts to the inequality 
\[
\begin{array}{ll}
\min\{x_{12},x_{13},x_{23},1-x_1-x_2-x_3+x_{12}+x_{13}+x_{23}\}\;\;\leq \\
\leq\;\; \max\{0,x_{12}+x_{13}-x_1,x_{12}+x_{23}-x_2,x_{13}+x_{23}-x_3\}\,.
\end{array}
\]
Then, by   Theorem \ref{THM:PIFOR3} it follows
\begin{corollary}\label{COR:PIFOR3}
For any coherent assessment  $(x_1,x_2,x_3,x_{12},x_{13},x_{23})$ on
\[
\{E_1|H_1,E_2|H_2,E_3|H_3,(E_1|H_1)\wedge (E_2|H_2), (E_1|H_1)\wedge (E_3|H_3),(E_2|H_2)\wedge (E_3|H_3)\}\]
the extension $x_{123}$ on $(E_1|H_1)\wedge (E_2|H_2)\wedge (E_3|H_3)$ is coherent if and only if $x_{123}\in[x_{123}',x_{123}'']$, where
\begin{equation}\label{EQ:INECOR}
\begin{array}{ll}
x_{123}'=\max\{0,x_{12}+x_{13}-x_1,x_{12}+x_{23}-x_2,x_{13}+x_{23}-x_3\},\\
x_{123}''= \min\{x_{12},x_{13},x_{23},1-x_1-x_2-x_3+x_{12}+x_{13}+x_{23}\}.
\end{array}
\end{equation}
\end{corollary}	
\begin{proof}	
As shown in  (\ref{EQ:SYSTEMPISTATEMENT}), (see also (\ref{EQ:SYSTEMPIcorta}) in the Appendix \ref{SECT:APPENDIXA}), the  coherence of  $(x_1,x_2,x_3,x_{12},x_{13},x_{23},x_{123})$ amounts to  the condition
\[
\begin{array}{ll}
\min\{x_{12},x_{13},x_{23},1-x_1-x_2-x_3+x_{12}+x_{13}+x_{23}\}\;\leq \;x_{123}\;\leq \\
\leq\;\; \max\{0,x_{12}+x_{13}-x_1,x_{12}+x_{23}-x_2,x_{13}+x_{23}-x_3\}\,.
\end{array}
\]
Then, in particular, 
	the extension $x_{123}$ on $(E_1|H_1)\wedge (E_2|H_2)\wedge (E_3|H_3)$ is coherent if and only if $x_{123}\in[x_{123}',x_{123}'']$, where
\[
\begin{array}{ll}
 x_{123}'=\max\{0,x_{12}+x_{13}-x_1,x_{12}+x_{23}-x_2,x_{13}+x_{23}-x_3\},\\
x_{123}''= \min\{x_{12},x_{13},x_{23},1-x_1-x_2-x_3+x_{12}+x_{13}+x_{23}\}.
\end{array}
\]
\qed
\end{proof}

\subsection{The Case $H_1=H_2=H_3$}
We recall that in case of logical dependencies, the set of all coherent assessments may be smaller than that one associated with  the  case of logical independence. 
However, in this section we show that the results of Theorem~\ref{THM:PIFOR3} and Corollary~\ref{COR:PIFOR3} still hold when the conditioning events $H_1,H_2$, and $H_3$ coincide.
\begin{theorem}\label{THM:PIFOR3bis}
Let be given any logically independent events $E_1, E_2, E_3,H$, with $H\neq \emptyset$.
Then, the set $\Pi$ of all coherent assessments  $\mathcal{M}=(x_1,x_2,x_3,x_{12},x_{13},x_{23},x_{123})$ on
$\F=\{E_1|H,E_2|H,E_3|H, (E_1|H)\wedge (E_2|H),$ $(E_1|H)\wedge (E_3|H),
(E_2|H)\wedge (E_3|H)$, ${(E_1|H)\wedge (E_2|H)\wedge (E_3|H)}\}$ is the set of points $(x_1,x_2,x_3,x_{12},x_{13},x_{23},x_{123})$ 
which satisfy the  conditions in formula (\ref{EQ:SYSTEMPISTATEMENT}).
\end{theorem}
\begin{proof}
	See Appendix \ref{SECT:APPENDIXA}.
\end{proof}	
\begin{corollary}\label{COR:PIFOR3bis}
	For any coherent assessment  $(x_1,x_2,x_3,x_{12},x_{13},x_{23})$ on
	\[
	\{E_1|H,E_2|H,E_3|H, (E_1E_2)|H,(E_1E_3)|H,
	(E_2E_3)|H\}\]
	the extension $x_{123}$ on $(E_1E_2E_3)|H$ is coherent if and only if $x_{123}\in[x_{123}',x_{123}'']$, where
$x_{123}'$ and  $x_{123}''$ are defined in 	(\ref{EQ:INECOR}).
\end{corollary}	
\begin{proof}
	The proof is the same as for Corollary \ref{COR:PIFOR3}. 
	\qed
\end{proof}	
Of course, the results of Theorem \ref{THM:PIFOR3bis} and  Corollary \ref{COR:PIFOR3bis} still hold in the unconditional case where ${H=\Omega}$.
\begin{remark}
As shown in this section, a consistent management of conjunctions (and/or disjunctions)  defined on a given family of conditional events $\F$  essentially requires  an (iterative) coherence checking and propagation of probability and prevision assessments on compounded  conditionals,  for each subfamily of $\F$.  Then,  an analysis of complexity in our context would be of the same kind of the   exhaustive complexity analysis  given in \cite{biazzo05} for probabilistic reasoning under coherence.
\end{remark}
\section{Characterization of p-consistency and p-entailment with applications to nonmonotonic reasoning}
\label{SECT:APPL}
In this section we apply our notion of conjunction to characterize the notions of p-consistency and p-entailment. Then, we examine some inference rules  related with probabilistic nonmonotonic reasoning.
We also briefly describe a  characterization of p-entailment   by a notion  of iterated conditioning,
in the case of two premises.  
We recall below the notions of p-consistency and p-entailment of Adams
(\cite{adams75}) as formulated for conditional events in
the setting of coherence  (see, e.g., \cite{biazzo05,gilio10,GiSa13IJAR}).  
\begin{definition}
	\label{PC}
	Let $\mathcal{F} = \{E_{i}|H_{i} \, , \; i=1,\ldots ,n\}$ be a
	family of $n$ conditional events. Then, $\mathcal{F}$ is \emph{p-consistent}
	if and only if the probability assessment $(p_{1},p_{2},\ldots ,
	p_{n})=(1,1,\ldots ,1)$ on $\mathcal{F}$ is coherent.
\end{definition}
\begin{definition}
	\label{PE}
	A p-consistent family $\mathcal{F} = \{E_{i}|H_{i} \, , \; i=1,\ldots ,n\}$
	\emph{p-entails} a conditional event $E_{n+1}|H_{n+1}$  if and only if for any coherent probability
	assessment $(p_{1},\ldots ,
	p_{n},p_{n+1})$ on $\mathcal{F} \cup \{E_{n+1}|H_{n+1}
	\}$ it holds that: if $p_{1}=\cdots =p_{n}=1$, then $p_{n+1}=1$.
\end{definition}
We  recall below  the  notion of logical implication (\cite{GoNg88}) between two conditional events. 
\begin{definition}\label{DEF:GN}
Given two conditional events $A|H$ and $B|K$ we say that \emph{$A|H$ logically implies $B|K$}, which we denote by $A|H \subseteq B|K$, if and only if
 		$AH$  \emph{true} implies $BK$  \emph{true} and
 		$\no{B}K$  \emph{true} implies $\no{A}H$  
 		\emph{true}; that is: $AH\subseteq BK$ and $\no{B}K\subseteq \no{A}H$.
\end{definition}
We observe that, by coherence, it holds that (see, e.g., \cite[Theorem 7]{gilio13}).
\begin{equation}\label{EQ:NGPROB} 
A|H\subseteq B|K\;\; \Longrightarrow \;\; P(A|H)\leq P(B|K).   
\end{equation}
We also recall the notion  of quasi conjunction for a general family of $n$ conditional events.
\begin{definition}\label{DEF:QC}
Given a family  $\mathcal{F} = \{E_{i}|H_{i} \, , \; i=1,\ldots ,n\}$ of $n$ conditional events, 
the \emph{quasi conjunction} $QC(\mathcal{F})$ of the conditional events in $\mathcal{F}$ is defined as the following conditional event
\[
QC(\mathcal{F})= \bigwedge_{i=1}^n(\no{H}_i\vee E_iH_i)| (\bigvee_{i=1}^n H_i).
\]	
\end{definition}
\begin{remark}
We observe that, by Definition \ref{DEF:QC},  based on (\ref{EQ:ESSE}) the quasi conjunction can be represented as 
 		\begin{equation}\label{EQ:QC}
 		\begin{array}{lll}
 		QC(\mathcal{F})=\sum_{h=0}^{m} \nu_{h}C_h, & \mbox{where} &
 		\nu_{h}=
 		\left\{
 		\begin{array}{llll}
 		1, &\mbox{ if } S_h'\neq \emptyset \mbox{ and }  S_h''=\emptyset,\\
 		0, &\mbox{ if } S_h''\neq \emptyset \\
 		\nu,   &\mbox{ if }  S_h''' =\{1,\ldots,n\}\,, 
 		\end{array}
 		\right.
 		\end{array}
 		\end{equation}
 where $\nu=P(QC(\F))$.
Therefore, 
 by (\ref{EQ:CF}), (\ref{EQ:QC}), and by also recalling Theorem \ref{THM:INEQ-CRQ}, 
  it holds that $z_h\leq \nu_h$, $h=0,1,\ldots,m$; thus
\begin{equation}\label{EQ:CFlessQC}
\C(\F)\leq QC(\F).
\end{equation}
In particular, if $\F$ is p-consistent and $P(E_i|H_i)=1$, $i=1,\ldots,n$, then from (\ref{EQ:LUKMIN}) it holds that $x_{S}=\prev(\C(\F_S))=1$ for every $S\subseteq\{1,2,\ldots,n\}$, where $\F_S=\{E_i|H_i\in \F: i \in S\}$;
then  $z_h=\nu_h$, $h=0,1,\ldots,m$, and  $\C(\F)=QC(\F)$. 
\end{remark}
\subsection{Characterization of p-consistency and p-entailment}
We illustrate below a characterization of p-consistency of a family $\F$ in terms of the coherence of the prevision assessment  $\prev[\C(\F)]=1$.
\begin{theorem}\label{THM:PCC}
A  family of $n$ conditional events $\F=\{E_1|H_1,\ldots, E_n|H_n\}$ is p-consistent if and only if the prevision assessment  $\prev[\C(\F)]=1$ is coherent.
\end{theorem}
\begin{proof}
$(\Rightarrow)$
 By Definition \ref{PC}, as $\F$ is p-consistent, the 
 probability assessment $(x_{1},x_{2},\ldots ,
	x_{n})=(1,1,\ldots ,1)$ on $\mathcal{F}$ is coherent. Then, by (\ref{EQ:LUKMIN}) the  extension $\prev[\C(\F)]=1$ is unique and of course  coherent.\\
$(\Leftarrow)$ By (\ref{EQ:LUKMIN}) it holds that $\prev[\C(\F)]\leq \min\{x_1,\ldots,x_n\}$
and hence ${\prev[\C(\F)]=1}$ implies  $(x_{1},x_{2},\ldots,x_{n})=(1,1,\ldots ,1)$ on $\mathcal{F}$. Moreover, the coherence of  $\prev[\C(\F)]=1$ requires that the  (unique) extension $(1,1,\ldots ,1)$ on $\F$ be coherent. Thus,  $\F$ is p-consistent.
\end{proof}
We observe that, in the case where $H_1=\ldots=H_n=H$, the assessment $P(E_1|H)=\ldots P(E_n|H)=1$ is coherent (that is, $\F$ is p-consistent) if and only if  $P[(E_1\cdots E_n)|H]=1$ is coherent.\\
The next theorem gives a characterization of p-entailment in terms of a result which involves suitable  conjunctions associated with the premise set and the conclusion of the given inference rule.
\begin{theorem}\label{THM:PENT}
Let be given a p-consistent family of $n$ conditional events $\F=\{E_1|H_1,\ldots, E_n|H_n\}$ and a further conditional event $E_{n+1}|H_{n+1}$. Then, the following assertions are equivalent:\\
(i) $\F$ p-entails $E_{n+1}|H_{n+1}$;\\
(ii) the conjunction $\C_{n+1}=(E_1|H_1)\wedge\cdots \wedge(E_n|H_n)\wedge (E_{n+1}|H_{n+1})$ coincides with
the  conjunction $\C_{n}=(E_1|H_1)\wedge\cdots \wedge(E_n|H_n)$;\\
(iii) the inequality  $\C_{n}\leq (E_{n+1}|H_{n+1})$ is satisfied.
\end{theorem}
\begin{proof}
	See Appendix \ref{SECT:APPENDIXA}.
\end{proof}	
As a first simple application of Theorem \ref{THM:PENT} we observe that, given two conditional events $A|H$, with $AH\neq \emptyset$, and $B|K$, the p-entailment of $B|K$  from $A|H$ amounts to the condition $(ii)$, i.e., $A|H\wedge B|K=A|H$, or equivalently condition $(iii)$, i.e.,  $A|H\leq B|K$. In particular, $(ii)$ and $(iii)$ are both satisfied  when $A|H\subseteq B|K$.
\subsection{Applications to some p-valid inference rules} 
We recall that  an inference from a p-consistent family $\mathcal{F}$ to $E|H$ is \emph{p-valid} if and only	if 	$\mathcal{F}$ p-entails $E|H$.
We will examine some p-valid inference rules by verifying that conditions $(ii)$ and $(iii)$ in Theorem~\ref{THM:PENT} are satisfied. In particular we consider the following inference rules of System P: \emph{And}, \emph{Cut}, \emph{CM}, and \emph{Or}. In what follows, if not specified otherwise, the basic events are assumed to be logically independent.
\paragraph{And rule:}  The family  $\{B|A, C|A\}$ p-entails  $BC|A$.
It holds that 
 $(B|A)\wedge (C|A)=BC|A=(B|A)\wedge (C|A) \wedge (BC|A)$ and  $(B|A)\wedge (C|A)=BC|A\leq BC|A$; that is,  conditions $(ii)$ and $(iii)$  are satisfied.
\paragraph{Cut rule:}
The family  $\{C|AB, B|A\}$ p-entails  $C|A$.
By  (\ref{EQ:REPRES}), as $\no{AB}AB=\no{A}ABC=\emptyset$, it holds that
\[
(C|AB)\wedge (B|A)=ABC+z\no{A},
\]
 where $z=\prev[(C|AB)\wedge (B|A)]$. 
Moreover, $BC|A=ABC+x\no{A}$, where $x=P(BC|A)$.
As $(C|AB)\wedge (B|A)$ and $BC|A$ coincide conditionally on $A$ being true, 
by Theorem \ref{THM:EQ-CRQ}, it follows that  $(C|AB)\wedge (B|A)=BC|A$.
Then, condition $(ii)$ is satisfied, that is $(C|AB)\wedge  (B|A) \wedge (C|A)=(BC|A) \wedge  (C|A)=BC|A=(C|AB)\wedge  (B|A)$. Moreover,  $C|AB\wedge  B|A=BC|A\leq C|A$, that is condition $(iii)$ is satisfied too.
\begin{remark}
As shown in the analysis of Cut rule, it holds that $C|AB\wedge  B|A=BC|A$. Then,  the family $\{C|AB,  B|A\}$ p-entails $BC|A$. This p-valid rule is
called CCT (Conjunctive Cumulative Transitivity); see, e.g.,  \cite{Verheij2017}.
\end{remark}	
\paragraph{CM rule:}
The family  $\{C|A, B|A\}$ p-entails  $C|AB$.
 It holds that $(C|A)\wedge (B|A)=BC|A$.
 Moreover, $(C|A)\wedge (B|A) \wedge (C|AB)=(BC|A) \wedge (C|AB)$. 
 By  (\ref{EQ:REPRES}),  it holds that
 \[
(BC|A)\wedge (C|AB)=ABC+z\no{A},
 \]
 where $z=\prev[(BC|A)\wedge (C|AB)]$. 
Moreover, $BC|A=ABC+x\no{A}$, where $x=P(BC|A)$.
 As $(BC|A)\wedge (C|AB)$ and $BC|A$ coincide conditionally on $A$ being true, 
 by Theorem \ref{THM:EQ-CRQ} it follows that  $(BC|A)\wedge (C|AB)=BC|A$; so that
 $(C|A)\wedge (B|A) \wedge (C|AB)=BC|A=(C|A)\wedge (B|A)$, so that condition $(ii)$ is satisfied.
Moreover,  based on Definition \ref{DEF:GN}, it holds that
 $(C|A)\wedge (B|A)=BC|A\subseteq C|AB$, then $(C|A)\wedge (B|A)\leq C|AB$, so that condition $(iii)$ is satisfied too.
\paragraph{Or rule:} The family  $\{C|A, C|B\}$ p-entails $C|(A \vee B)$. 
 We set $P(C|A)=x$, $P(C|B)=y$, and $\prev((C|A) \wedge (C|B))=z$; then,  by observing that the family $\{ABC,\no{A}BC,A\no{B}C,(A\vee B)\no{C},\no{A}\no{B}\}$ is a partition of the sure event, we obtain 
 \begin{equation}
 \label{EQ:CAandCB}
 (C|A) \wedge (C|B) = \left\{\begin{array}{ll}
 1, &\mbox{if $ABC$ is true,}\\
 0, &\mbox{if $(A\vee B)\no{C}$ is true,}\\
 x, &\mbox{if $\no{A}BC$ is true,}\\
 y, &\mbox{if $A\no{B}C$ is true,}\\
 z, &\mbox{if $\no{A}\no{B}$ is true.}
 \end{array}
 \right.
 \end{equation}
 Moreover, by defining $\prev[  (C|A) \wedge (C|B)\wedge (C|(A \vee B))]=t$, we obtain
 \[
  (C|A) \wedge (C|B)\wedge (C|(A \vee B))= 
 \left\{\begin{array}{ll}
 1, &\mbox{if $ABC$ is true,}\\
 0, &\mbox{if $(A\vee B)\no{C}$ is true,}\\
 x, &\mbox{if $\no{A}BC$ is true,}\\
 y, &\mbox{if $A\no{B}C$ is true,}\\
 t, &\mbox{if $\no{A}\no{B}$ is true.}
 \end{array}
 \right.
 \]
 As we can see, $  (C|A) \wedge (C|B)\wedge (C|(A \vee B))$ and $(C|A) \wedge (C|B)$ coincide when $A \vee B$ is true; then, by Theorem \ref{THM:EQ-CRQ}  it holds that $t = z$, so that 
 \[
  (C|A) \wedge (C|B)\wedge (C|(A \vee B)) = (C|A) \wedge (C|B),
 \]
that is condition $(ii)$ is satisfied.
Moreover, defining $P(C|(A \vee B))=w$, we have 
 \begin{equation}\label{EQ:CgAvB}
C|(A\vee B)=
\left\{\begin{array}{ll}
 1, &\mbox{if $ABC$ is true,}\\
 0, &\mbox{if $(A\vee B)\no{C}$ is true,}\\
 1, &\mbox{if $\no{A}BC$ is true,}\\
 1, &\mbox{if $A\no{B}C$ is true,}\\
w , &\mbox{if $\no{A}\no{B}$ is true.}
 \end{array}
 \right.
 \end{equation}
Based on  (\ref{EQ:CAandCB}) and (\ref{EQ:CgAvB}),  it holds that $ (C|A) \wedge (C|B) \leq C|(A\vee B)$ conditionally on $A \vee B$ being true. Then, from Theorem \ref{THM:INEQ-CRQ} it holds that $\prev((C|A) \wedge (C|B))=t\leq w=P(C|(A\vee B))$; thus 
 $ (C|A) \wedge (C|B) \leq C|(A\vee B)$, that is condition $(iii)$ is satisfied.
 \paragraph{An inference rule related to Or rule \cite[Rule 5, p. 189]{adams65}.} In this inference rule  the premise set is  $\{C|(A\vee B), \no{C}|A\}$ and the conclusion is $C|B$. 
We first observe that the premise set $\F=\{C|(A\vee B), \no{C}|A\}$ is p-consistent because the assessment $P(C|(A\vee B))=P(\no{C}|A)=1$ is coherent.  Indeed, by applying Algorithm \ref{ALG-PREV-INT} to the pair $(\F,\M)=(\{C|(A\vee B), \no{C}|A\},(1,1))$, it holds that the starting system $(\Sigma)$ is solvable, with $\F_0=\{\no{C}|A\}$. Then, by repeating the steps of the algorithm,  the assessment $P(\no{C}|A)=1$ is coherent. Thus, the assessment $(1,1)$ on $\F$ is coherent and hence  $\F$ is p-consistent.
We also note that, defining
 $P(C|(A\vee B))=x$, $P(\no{C}|A)=y$, and $\prev((C|(A\vee B)) \wedge (\no{C}|A))=z$, 
 the coherence of $(x,y)=(1,1)$ from (\ref{EQ:LUKMIN}) amounts to coherence of $z=1$, which  by Theorem \ref{THM:PCC} is another characterization for the  p-consistency of $\F$. Concerning p-entailment, we observe that 
 \begin{equation}
 \label{EQ:CAvBandnCA}
(C|A\vee B) \wedge (\no{C}|A) =
\left\{\begin{array}{ll}
0, &\mbox{if $ABC$ is true,}\\	
0, &\mbox{if $AB\no{C}$ is true,}\\
0, &\mbox{if $A\no{B}C$ is true,}\\
0, &\mbox{if $A\no{B}\no{C}$ is true,}\\
y, &\mbox{if $\no{A}BC$ is true,}\\	
0, &\mbox{if $\no{A}B\no{C}$ is true,}\\
z, &\mbox{if $\no{A}\no{B}$ is true,}\\	
\end{array}
\right.
=
 \left\{\begin{array}{ll}
 0, &\mbox{if $A\vee \no{A}B\no{C}$ is true,}\\
 y, &\mbox{if $\no{A}BC$ is true,}\\
 z, &\mbox{if $\no{A}\no{B}$ is true.}
 \end{array}
 \right. 
 \end{equation}
 Moreover, by defining $\prev[  (C|(A\vee B)) \wedge (\no{C}|A)  \wedge (C|B)]=t$, we obtain
  \begin{equation}
 \label{EQ:CAvBandnCAandCB}
 (C|(A\vee B)) \wedge (\no{C}|A)  \wedge (C|B) = \left\{\begin{array}{ll}
 0, &\mbox{if $A\vee \no{A}B\no{C}$ is true,}\\
 y, &\mbox{if $\no{A}BC$ is true,}\\
 t, &\mbox{if $\no{A}\no{B}$ is true.}
 \end{array}
 \right.
 \end{equation}
  As we can see from (\ref{EQ:CAvBandnCA}) and  (\ref{EQ:CAvBandnCAandCB}), the two quantities $(C|(A\vee B)) \wedge (\no{C}|A)  \wedge (C|B)$ and $(C|(A\vee B)) \wedge (\no{C}|A)$ coincide when $A \vee B$ is true; then, by Theorem \ref{THM:EQ-CRQ}  it holds that $t = z$, so that 
 \[
  (C|(A\vee B)) \wedge (\no{C}|A)  \wedge (C|B) = (C|(A\vee B)) \wedge (\no{C}|A),
 \]
 that is condition $(ii)$ is satisfied.
 Moreover, defining $P(C|B)=w$, we have 
  \begin{equation}
\label{EQ:CB}
C|B = 
\left\{\begin{array}{ll}
1, &\mbox{if $BC$ is true,}\\
0, &\mbox{if $B\no{C}$ is true,}\\
w, &\mbox{if $\no{B}$ is true.}
\end{array}
\right.=\left\{\begin{array}{ll}
1, &\mbox{if $ABC$ is true,}\\	
0, &\mbox{if $AB\no{C}$ is true,}\\
w, &\mbox{if $A\no{B}C$ is true,}\\
w, &\mbox{if $A\no{B}\no{C}$ is true,}\\
1, &\mbox{if $\no{A}BC$ is true,}\\	
0, &\mbox{if $\no{A}B\no{C}$ is true,}\\
w, &\mbox{if $\no{A}\no{B}$ is true,}\\	
\end{array}
\right.
  \end{equation}
Based on  (\ref{EQ:CAvBandnCA}) and (\ref{EQ:CB}),  it holds that $(C|(A\vee B)) \wedge (\no{C}|A) \leq C|B$ conditionally on $A \vee B$ being true. Then, from Theorem \ref{THM:INEQ-CRQ} it holds that $\prev((C|(A\vee B)) \wedge (\no{C}|A))=t\leq w=P(C|B)$; thus 
 $ (C|(A\vee B)) \wedge (\no{C}|A) \leq C|B$, that is condition $(iii)$ is satisfied.
Thus, this inference rule is p-valid. Notice that the p-validity of the rule could be also derived by using the lower and upper bounds given for Or rule in \cite{gilio02}. Indeed, using Or rule, when $P(C|A)=0$ and $P(C|B)=y$ it holds that  $z=P(C|A\vee B)\in[0, y]$, so that $P(C|(A\vee B))\leq P(C|B)$.  Then, $P(C|(A\vee B))=1$ and $P(\no{C}|A)=1$ implies $P(C|B)=1$, that is $\{C|(A\vee B), \no{C}|A\}$ p-entails $C|B$.
\paragraph{Generalized Or rule:} 
In this p-valid rule, studied in \cite{gilio12ijar} (see also \cite{gilio13}), the p-consistent premise set  is $\{C|A_1, C|A_2,\ldots, C|A_n\}$ and the conclusion is  $C|(A_1 \vee A_2\vee \cdots \vee A_n)$. 
For each nonempty subset $S\subset \{1,2,\ldots,n\}$, we define   
$\prev[\bigwedge_{i\in S}(C|A_i)]=x_S$; moreover, we set $\prev[\bigwedge_{i=1}^n(C|A_i)]=z$.  Then,
\begin{equation}
\label{EQ:CAandCBGEN}
(C|A_1) \wedge\cdots \wedge (C|A_n)= \left\{\begin{array}{ll}
1, &\mbox{if $A_1A_2\cdots A_nC$ is true,}\\
0, &\mbox{if $(A_1\vee A_2\vee \cdots \vee A_n)\no{C}$ is true,}\\
x_S, &\mbox{if $ \bigwedge_{i\in S}\no{A}_i  \bigwedge_{j\notin S}{A_j}C$ is true,}\\
z, &\mbox{if $\no{A}_1\no{A}_2\cdots \no{A}_n$ is true}.
\end{array}
\right.
\end{equation}
Moreover, by defining $\prev[  (C|A_1) \wedge\cdots \wedge (C|A_n)\wedge (C|(A_1 \vee A_2\vee \cdots \vee A_n))]=t$, we obtain
\begin{equation}\label{EQ:29}
(C|A_1) \wedge\cdots \wedge (C|A_n)\wedge (C|(A_1 \vee A_2\vee \cdots \vee A_n))= 
\left\{\begin{array}{ll}
1, &\mbox{if $A_1A_2\cdots A_nC$ is true,}\\
0, &\mbox{if $(A_1\vee A_2\vee \cdots \vee A_n)\no{C}$ is true,}\\
x_S, &\mbox{if $ \bigwedge_{i\in S}\no{A}_i  \bigwedge_{j\notin S}{A_j}C$ is true,}\\
t, &\mbox{if $\no{A}_1\no{A}_2\cdots \no{A}_n$ is true}.
\end{array}
\right.
\end{equation}
As we can see from (\ref{EQ:CAandCBGEN}) and  (\ref{EQ:29}), $ (C|A_1) \wedge\cdots \wedge (C|A_n)\wedge (C|(A_1 \vee A_2\vee \cdots \vee A_n))$ and $(C|A_1) \wedge\cdots \wedge (C|A_n)$ coincide when $A_1 \vee \cdots \vee A_n$ is true; then, by Theorem \ref{THM:EQ-CRQ}  it holds that $t = z$, so that 
\[
(C|A_1) \wedge\cdots \wedge (C|A_n)\wedge (C|(A_1 \vee A_2\vee \cdots \vee A_n)) = (C|A_1) \wedge\cdots \wedge (C|A_n),
\]
that is condition $(ii)$ is satisfied.
Moreover, 
\begin{equation}
\label{EQ:CgAvBGEN}
C|(A_1 \vee A_2\vee \cdots \vee A_n)=
\left\{\begin{array}{ll}
1, &\mbox{if $A_1A_2\cdots A_nC$ is true,}\\
0, &\mbox{if $(A_1\vee A_2\vee \cdots \vee A_n)\no{C}$ is true,}\\
1, &\mbox{if $ \bigwedge_{i\in S}\no{A}_i  \bigwedge_{j\notin S}{A_j}C$ is true,}\\
w, &\mbox{if $\no{A}_1\no{A}_2\cdots \no{A}_n$ is true},
\end{array}
\right.
\end{equation}
where $w=P(C|(A_1 \vee A_2\vee \cdots \vee A_n))$. 
Based on  (\ref{EQ:CAandCBGEN}) and (\ref{EQ:CgAvBGEN}),  it holds that 
$ (C|A_1) \wedge\cdots \wedge (C|A_n)\leq C|(A_1 \vee A_2\vee \cdots \vee A_n)$ 
conditionally on $A_1 \vee \cdots \vee A_n$ being true. Then, from Theorem \ref{THM:INEQ-CRQ} it holds that $t\leq w$; thus 
$(C|A_1) \wedge\cdots \wedge (C|A_n)\leq C|(A_1 \vee A_2\vee \cdots \vee A_n)$, that is condition $(iii)$ is satisfied.\\
\subsection{Iterated conditioning and  p-entailment}
We now briefly describe a  characterization of p-entailment  of a conditional event $E_3|H_3$ from a p-consistent family $\{E_1|H_1, E_2|H_2\}$,  which exploits  a suitable notion of iterated conditioning. 
\begin{definition}\label{DEF:GENITER}
	Let be given  $n+1$ conditional events $E_1|H_1, \ldots, E_{n+1}|H_{n+1}$, with $(E_1|H_1) \wedge \cdots \wedge (E_n|H_n)\neq 0$. We denote by $(E_{n+1}|H_{n+1})|((E_1|H_1) \wedge \cdots \wedge (E_n|H_n))=(E_{n+1}|H_{n+1})|\C_n$ the random quantity
	\[
	\begin{array}{ll}
	(E_1|H_1) \wedge \cdots \wedge (E_{n+1}|H_{n+1}) + \mu (1-(E_1|H_1) \wedge \cdots \wedge (E_n|H_n)) =\\
	=\C_{n+1}+\mu (1-\C_n),
	\end{array}
	\]
	where $\mu = \prev[(E_{n+1}|H_{n+1})|\C_n]$.
\end{definition}
We observe that, based on the betting metaphor,  the quantity $\mu$ is the amount to be paid in order to receive the amount $\C_{n+1}+\mu (1-\C_n)$. Definition \ref{DEF:GENITER} generalizes the notion  of iterated conditional $(E_{2}|H_{2})|(E_1|H_1)$ given in 
previous papers (see, e.g., \cite{GiSa13c,GiSa13a,GiSa14}). 
We also observe that, defining 	$\prev(\C_n)=z_{n}$ and 
$\prev(\C_{n+1})=z_{n+1}$, by the linearity of prevision it holds that $\mu=z_{n+1}+\mu(1-z_{n})$; then, $z_{n+1}=\mu z_n$, that is 
$\prev(\C_{n+1})=\prev[(E_{n+1}|H_{n+1})|\C_n]\prev(\C_n)$, which is  the compound prevision theorem.

By applying Definition \ref{DEF:GENITER} with $n=2$, given a 
p-consistent family $\{E_1|H_1, E_2|H_2\}$ and a further event $E_3|H_3$,  it can  be proved that (\cite{GiPS18wp})
\[
\{E_1|H_1, E_2|H_2\} \mbox{ p-entails } E_3|H_3\;\; \Longleftrightarrow \;\;(E_3|H_3)|((E_1|H_1)\wedge(E_2|H_2))=1\,,
\]
that is: $\{E_1|H_1, E_2|H_2\}$ p-entails $E_3|H_3$ if and only if  the iterated conditional $(E_3|H_3)|((E_1|H_1)\wedge(E_2|H_2))$ is constant and  equal to~1. 
%
%
%
\section{From  non p-valid to p-valid inference rules}
\label{SECT:NPV-PV}
In this section we first  examine some non p-valid inference rules, by showing that conditions $(ii)$ and $(iii)$ of Theorem \ref{THM:PENT} are not satisfied.   Then, we  illustrate by an example two different methods which allow to  get p-valid inference rules starting by non p-valid ones.
\subsection{Some non p-valid inference rules}
 We start by showing that Transitivity is not p-valid.
\paragraph{Transitivity.}
In this rule  the p-consistent premise set is $\{C|B,B|A\}$ and the conclusion is 
$C|A$. The rule is not p-valid (\cite{gilio16}), indeed we can show that 
\[
(C|B)\wedge(B|A)\wedge (C|A)\neq (C|B)\wedge(B|A) \mbox{ and } (C|B)\wedge(B|A)\nleq C|A.
\] 
Defining $P(B|A)=x$, $P(BC|A)=y$, $P(C|A)=t$, $\prev[(C|B)\wedge(B|A)\wedge (C|A)]=\mu$, $\prev[(C|B)\wedge(B|A)]=z$, 
we have 
\begin{equation}\label{EQ:TR1}
(C|B)\wedge(B|A)\wedge (C|A)=(C|B)\wedge(BC|A)=
\left\{\begin{array}{ll}
1, &\mbox{if $ABC$ is true,}\\	
0, &\mbox{if $AB\no{C}$ is true,}\\
0, &\mbox{if $A\no{B}C$ is true,}\\
0, &\mbox{if $A\no{B}\no{C}$ is true,}\\
y, &\mbox{if $\no{A}BC$ is true,}\\	
0, &\mbox{if $\no{A}B\no{C}$ is true,}\\
\mu, &\mbox{if $\no{A}\no{B}$ is true,}\\	
\end{array}
\right.
\end{equation} 
and
\begin{equation}\label{EQ:TR2}
(C|B)\wedge(B|A)=
\left\{\begin{array}{ll}
1, &\mbox{if $ABC$ is true,}\\	
0, &\mbox{if $AB\no{C}$ is true,}\\
0, &\mbox{if $A\no{B}C$ is true,}\\
0, &\mbox{if $A\no{B}\no{C}$ is true,}\\
x, &\mbox{if $\no{A}BC$ is true,}\\	
0, &\mbox{if $\no{A}B\no{C}$ is true,}\\
z, &\mbox{if $\no{A}\no{B}$ is true.}\\	
\end{array}
\right.
\end{equation} 
Then, as (in general) $x\neq y$, it holds that $(C|B)\wedge(B|A)\wedge (C|A)\neq (C|B)\wedge(B|A)$, so that  condition $(ii)$ is not satisfied. 
Moreover, 
\begin{equation}\label{EQ:TR3}
C|A=
\left\{\begin{array}{ll}
1, &\mbox{if $ABC$ is true,}\\	
0, &\mbox{if $AB\no{C}$ is true,}\\
0, &\mbox{if $A\no{B}C$ is true,}\\
0, &\mbox{if $A\no{B}\no{C}$ is true,}\\
t, &\mbox{if $\no{A}BC$ is true,}\\	
t, &\mbox{if $\no{A}B\no{C}$ is true,}\\
t, &\mbox{if $\no{A}\no{B}$ is true.}\\	
\end{array}
\right.
\end{equation} 
Then, by observing  that (in general) $x\nleq t$ it follows that 
$(C|B)\wedge(B|A)\nleq C|A$, 
so that  condition $(iii)$ is not satisfied. 
Therefore, Transitivity rule is not p-valid.
\paragraph{Denial of the antecedent.}
We consider the rule where the premise set is $\{\no{A},C|A\}$ and the conclusion is $\no{C}$. 
The premise set  $\{\no{A},C|A\}$ is p-consistent because, by applying the Algorithm \ref{ALG-PREV-INT}, the assessment $P(\no{A})=P(C|A)=1$ is coherent. We verify that $\no{A}\wedge (C|A)\wedge \no{C}\neq \no{A}\wedge (C|A)$ and that $\no{A}\wedge (C|A)\nleq \no{C}$, that is  the Denial of the antecedent is not p-valid.
We set $P(C|A)=y$, then
\[
\no{A}\wedge (C|A)\wedge\no{C}=
 \left\{\begin{array}{ll}
	0, &\mbox{if $A$ is true,}\\
	0, &\mbox{if $\no{A}C$ is true,}\\
	y, &\mbox{if $\no{A}\no{C}$ is true,}\\
	\end{array}
	\right.
\] 
and
\[
\no{A}\wedge (C|A)=
 \left\{\begin{array}{ll}
	0, &\mbox{if $A$ is true,}\\
	y, &\mbox{if $\no{A}C$ is true,}\\
	y, &\mbox{if $\no{A}\no{C}$ is true.}\\
	\end{array}
	\right.
\] 
Assuming $y>0$, when $\no{A}C$ is true it holds that   
\[
\no{A}\wedge (C|A)\wedge \no{C}=0< y= \no{A}\wedge (C|A),\;\;\;\no{A}\wedge (C|A)=y> 0=\no{C},
\]
thus: $\no{A}\wedge (C|A)\wedge \no{C}\neq \no{A}\wedge (C|A)$ and $\no{A}\wedge (C|A)\nleq \no{C}$, that is  conditions $(ii)$ and $(iii)$ are not satisfied. 
\paragraph{Affirmation of the consequent.}
We consider the rule where the (p-consistent) premise set is $\{C,C|A\}$ and the conclusion is $A$. We verify that  $C\wedge (C|A)\wedge A\neq C\wedge (C|A)$ and $C\wedge (C|A)\nleq A$, that is the \emph{Affirmation of the consequent} rule is not p-valid.
We set $P(C|A)=y$, then
\[
C\wedge (C|A)\wedge A=AC=
 \left\{\begin{array}{ll}
	1, &\mbox{if $AC$ is true,}\\
	0, &\mbox{if $\no{AC}$ is true,}\\
	\end{array}
	\right.
\] 
and
\begin{equation}
\label{EQ:CandCgA}
C\wedge (C|A)=
 \left\{\begin{array}{ll}
	1, &\mbox{if $AC$ is true,}\\
	0, &\mbox{if $\no{C}$ is true,}\\
	y, &\mbox{if $\no{A}C$ is true.}\\
	\end{array}
	\right.
\end{equation}
Assuming $y>0$, when $\no{A}C$ is true it holds that   
\[
C\wedge (C|A)\wedge A=0< y= C\wedge (C|A),\;\;\;C\wedge (C|A)=y> 0=A,
\]
thus: $C\wedge (C|A)\wedge A\neq C\wedge (C|A)$ and $C\wedge (C|A)\nleq A$, that is  conditions $(ii)$ and $(iii)$ are not satisfied.  
\begin{remark} We now will make a comparison between the two objects $C\wedge (C|A)$  and $C|(A\vee \no{C})$, by showing they do not coincide.
Defining $P(C|(A\vee \no{C}))=t$, it holds that
\begin{equation}\label{EQ:CgAnoC}
C|(A\vee \no{C} )=
	\left\{\begin{array}{ll}
		1, &\mbox{if $AC$ is true,}\\
		0, &\mbox{if $\no{C}$ is true,}\\
		t, &\mbox{if $\no{A}C$ is true.}\\
		\end{array}
		\right.
\end{equation}
It could seem, from (\ref{EQ:CandCgA}) and (\ref{EQ:CgAnoC}),  that $y$ and $t$ should be equal and then $C\wedge (C|A)$ and $C|(A\vee \no{C})$ should coincide. However, in this case the conditioning event for $C\wedge (C|A)$ is $\Omega \vee A=\Omega$, so that  the disjunction of the conditioning events is $\Omega \vee (A\vee \no{C})=\Omega$;  the two objects $C\wedge (C|A)$ and $C|(A\vee \no{C})$ do not coincide conditionally on $\Omega$; then   $C\wedge (C|A)$  and $C|(A\vee \no{C})$ do not coincide (condition $(i)$ of  Theorem \ref{THM:EQ-CRQ} is not satisfied). We also observe that,
defining  $\prev(C\wedge (C|A))=\mu$,  (in general) $\mu$ does not belong to the set   $\{1,0,y\}$ of possible values of $C\wedge (C|A)$, because $\mu$ is a linear convex combination of the values  $\{1,0,y\}$. As a  further aspect,  we verify below that $t\leq \mu\leq y$. 
The constituents generated by $\{A,C\}$ are: $AC, A\no{C}, \no{A}C, \no{A}\no{C}$;  then, the associated values for the random vector $(C|(A \vee \no{C}), C \wedge (C|A), C|A)$ are
\begin{equation}\label{EQ:SETVALUES}
(1,1,1) \,,\;\;\; (0,0,0) \,,\;\;\; (t,y,y) \,,\;\;\; (0,0,y).
\end{equation}
Based on Theorem \ref{THM:INEQ-CRQ}, we observe that:\\ 
$\bullet$ $C|(A \vee \no{C}) \leq C|A$ conditionally on $A \vee \no{C}\vee A=A \vee \no{C}$, 
 hence   $P(C|(A \vee \no{C}))=t \leq y=P(C|A)$; \\  
$\bullet$ $C \wedge (C|A) \leq C|A$ conditionally on $\Omega \vee A  =\Omega$, hence  $\prev(C \wedge (C|A))=\mu \leq y=P(C|A)$; \\ 
$\bullet$ $C|(A \vee \no{C}) \leq C \wedge (C|A)$ conditionally on $A \vee \no{C} \vee \Omega  =\Omega$, hence $P(C|(A \vee \no{C}))=t \leq \mu=\prev(C \wedge (C|A))$. \\
 In other words: $t\leq \mu\leq y$. We observe that these inequalities also follow because coherence requires that the prevision point $(t,\mu,y)$  must be a linear convex combination of 
 points in ($\ref{EQ:SETVALUES}$). 
\end{remark}
\paragraph{On combining evidence: An example from Boole.} 
We now examine an example studied in \cite[p. 632]{boole_1857} (see also \cite[Theorem 5.45]{hailperin96}), where p-entailment does not hold. Indeed, it can be proved that  the extension $w=P(C|AB)$ of any (coherent)  assessment  $(x,y)$ on $\{C|A,C|B\}$ is coherent for every $w\in[0,1]$. 
Using conditions $(ii)$ and $(iii)$ of Theorem \ref{THM:PENT}, we show that the p-consistent family  $\{C|A, C|B\}$ does not p-entail $C|AB$. 
 We set $P(C|A)=x$, $P(C|B)=y$, and $\prev((C|A) \wedge (C|B))=z$; then, 
 \begin{equation}\label{EQ:CgACgBbis}
 (C|A) \wedge (C|B) = \left\{\begin{array}{ll}
 	1, &\mbox{if $ABC$ is true,}\\	
 	0, &\mbox{if $AB\no{C}$ is true,}\\
 	0, &\mbox{if $A\no{B}\no{C}$ is true,}\\
 	0, &\mbox{if $\no{A}B\no{C}$ is true,}\\
 	x, &\mbox{if $\no{A}BC$ is true,}\\	
 	y, &\mbox{if $A\no{B}C$ is true,}\\
 	z, &\mbox{if $\no{A}\no{B}$ is true.}\\	
 	\end{array}
 	\right.= \left\{\begin{array}{ll}
 1, &\mbox{if $ABC$ is true,}\\
 0, &\mbox{if $(A\vee B) C$  is true,}\\
 x, &\mbox{if $\no{A}BC$ is true,}\\
 y, &\mbox{if $A\no{B}C$ is true,}\\
 z, &\mbox{if $\no{A}\no{B}$ is true.}
 \end{array}
 \right.
 \end{equation}
 Moreover, by defining $ \prev[(C|A)\wedge (C|AB)]=u$, $ \prev[(C|B)\wedge (C|AB)]=v$ and 
  $\prev[(C|A) \wedge (C|B)\wedge (C|AB) ]=t$, we obtain
 \[
(C|A) \wedge (C|B)\wedge (C|AB) = 
 \left\{\begin{array}{ll}
 1, &\mbox{if $ABC$ is true,}\\
 0, &\mbox{if $(A\vee B)\no{C}$ is true,}\\
u, &\mbox{if $\no{A}BC$ is true,}\\
v, &\mbox{if $A\no{B}C$ is true,}\\
 t, &\mbox{if $\no{A}\no{B}$ is true.}
 \end{array}
 \right.
 \]
 As in general $x\neq u$ and $y\neq v$, then $(C|A) \wedge (C|B)\wedge (C|AB)$ and $(C|A) \wedge (C|B)$ do not coincide, so that   condition $(ii)$ is not satisfied. Moreover, 
 \begin{equation}\label{EQ:CgAB}
C|AB=
\left\{\begin{array}{ll}
 	1, &\mbox{if $ABC$ is true,}\\	
 	0, &\mbox{if $AB\no{C}$ is true,}\\
 	w, &\mbox{if $A\no{B}\no{C}$ is true,}\\
 	w, &\mbox{if $\no{A}B\no{C}$ is true,}\\
 	w, &\mbox{if $\no{A}BC$ is true,}\\	
 	w, &\mbox{if $A\no{B}C$ is true,}\\
 	w, &\mbox{if $\no{A}\no{B}$ is true,}\\	
 	\end{array}
 	\right.
=\left\{\begin{array}{ll}
 1, &\mbox{if $ABC$ is true,}\\
 0, &\mbox{if $AB\no{C}$ is true,}\\
 w, &\mbox{if $\no{AB}$ is true.}\\
 \end{array}
 \right.
 \end{equation}
Based on  (\ref{EQ:CgACgBbis}) and (\ref{EQ:CgAB}),  we can see that $ (C|A) \wedge (C|B) \nleq C|(AB)$, so that   condition $(iii)$ is not satisfied.  Thus, the inference from $\{C|A, C|B\}$ to $C|AB$ is not p-valid.
\subsection{Two methods for constructing  p-valid inference rules}
We now illustrate by an example two different methods by means of which, starting by a non p-valid inference rule, we get p-valid inference rules: a) to add  a suitable premise; b) to add a suitable logical constraint. The further premise, or logical constraint, (must  preserve p-consistency and) is  determined by analyzing the possible values of conjunctions.
\paragraph{Weak Transitivity.}
In our example we start by the (non p-valid) Transitivity rule where the premise set is $\{C|B,B|A\}$ and the conclusion is $C|A$.\\
Method a). We add the premise $A|(A\vee B)$, so that 
 the premise set is $\{C|B,B|A, A|(A\vee B)\}$, while the conclusion is still
$C|A$. The premise set $\{C|B,B|A, A|(A\vee B)\}$ is p-consistent; indeed as $ABC\neq \emptyset$, by evaluating $P(ABC)=1$ we get $P(C|B)=P(B|A)=P(A|(A\vee B))=1$. 
 We  show that $(C|B)\wedge(B|A)\wedge (A|(A\vee B))\wedge (C|A)= (C|B)\wedge(B|A)\wedge (A|(A\vee B))$ and 
$(C|B)\wedge(B|A)\wedge (A|(A\vee B))\leq C|A$. 

Defining $\prev[(C|B)\wedge(B|A)\wedge (A|(A\vee B))\wedge (C|A)]=\mu$,  
we have 
\[
(C|B)\wedge(B|A)\wedge (A|(A\vee B))\wedge (C|A)=
\left\{\begin{array}{ll}
1, &\mbox{if $ABC$ is true,}\\	
0, &\mbox{if $AB\no{C}$ is true,}\\
0, &\mbox{if $A\no{B}C$ is true,}\\
0, &\mbox{if $A\no{B}\no{C}$ is true,}\\
0, &\mbox{if $\no{A}BC$ is true,}\\	
0, &\mbox{if $\no{A}B\no{C}$ is true,}\\
\mu, &\mbox{if $\no{A}\no{B}$ is true.}\\	
\end{array}
\right.
\] 
Moreover, defining $\prev[(C|B)\wedge(B|A)\wedge (A|(A\vee B))]=z$, we have 
\[
(C|B)\wedge(B|A)\wedge (A|(A\vee B))=
\left\{\begin{array}{ll}
1, &\mbox{if $ABC$ is true,}\\	
0, &\mbox{if $AB\no{C}$ is true,}\\
0, &\mbox{if $A\no{B}C$ is true,}\\
0, &\mbox{if $A\no{B}\no{C}$ is true,}\\
0, &\mbox{if $\no{A}BC$ is true,}\\	
0, &\mbox{if $\no{A}B\no{C}$ is true,}\\
z, &\mbox{if $\no{A}\no{B}$ is true.}\\	
\end{array}
\right.
\]
Conditionally on $A\vee B$ being true it holds that $(C|B)\wedge(B|A)\wedge (A|(A\vee B))\wedge (C|A)=(C|B)\wedge(B|A)\wedge (A|(A\vee B))=ABC|(A\vee B)$. Then,  by Theorem \ref{THM:EQ-CRQ} we have $(C|B)\wedge(B|A)\wedge (A|(A\vee B))\wedge (C|A)=(C|B)\wedge(B|A)\wedge (A|(A\vee B))=ABC|(A\vee B)$,
so that  condition $(ii)$ is satisfied. 
Finally, as $ABC|(A\vee B)\subseteq C|A$, it holds that 
$(C|B)\wedge(B|A)\wedge (A|(A\vee B))=ABC|(A\vee B)\leq C|A$,
so that  condition $(iii)$ is satisfied.   Therefore  this Weak Transitivity rule is p-valid. We observe that another p-valid version of Weak Transitivity  would be obtained by adding the premise  $A|B$ instead of $A|(A\vee B)$. \\
Method b).  We add the logical constraint $\no{A}BC=\emptyset$, that is $BC\subseteq A$.
The p-consistency of the premise set $\{C|B,B|A\}$ is preserved because, as before  $ABC\neq \emptyset$ and by evaluating $P(ABC)=1$ we get $P(C|B)=P(B|A)=1$. 
Based on (\ref{EQ:TR1}), (\ref{EQ:TR2}) , (\ref{EQ:TR3}) it holds that 
\begin{equation*}
(C|B)\wedge(B|A)\wedge (C|A)=(C|B)\wedge(BC|A)=
\left\{\begin{array}{ll}
1, &\mbox{if $ABC$ is true,}\\	
0, &\mbox{if $AB\no{C}$ is true,}\\
0, &\mbox{if $A\no{B}C$ is true,}\\
0, &\mbox{if $A\no{B}\no{C}$ is true,}\\
0, &\mbox{if $\no{A}B\no{C}$ is true,}\\
\mu, &\mbox{if $\no{A}\no{B}$ is true,}\\	
\end{array}
\right.
\end{equation*} 
and
\begin{equation*}
(C|B)\wedge(B|A)=
\left\{\begin{array}{ll}
1, &\mbox{if $ABC$ is true,}\\	
0, &\mbox{if $AB\no{C}$ is true,}\\
0, &\mbox{if $A\no{B}C$ is true,}\\
0, &\mbox{if $A\no{B}\no{C}$ is true,}\\
0, &\mbox{if $\no{A}B\no{C}$ is true,}\\
z, &\mbox{if $\no{A}\no{B}$ is true.}\\	
\end{array}
\right.
\end{equation*} 
As we can see  $(C|B)\wedge(B|A)\wedge (C|A)= (C|B)\wedge(B|A)$ conditionally on $A\vee B$ being true. Then, by Theorem \ref{THM:EQ-CRQ}   condition $(ii)$ is  satisfied. 
Moreover, 
\begin{equation*}
C|A=
\left\{\begin{array}{ll}
1, &\mbox{if $ABC$ is true,}\\	
0, &\mbox{if $AB\no{C}$ is true,}\\
0, &\mbox{if $A\no{B}C$ is true,}\\
0, &\mbox{if $A\no{B}\no{C}$ is true,}\\
t, &\mbox{if $\no{A}BC$ is true,}\\	
t, &\mbox{if $\no{A}\no{B}$ is true.}\\	
\end{array}
\right.
\end{equation*} 
Then, $(C|B)\wedge(B|A)\leq C|A$ conditionally on $A\vee B$ being true.  Thus, by Theorem \ref{THM:INEQ-CRQ} 
 condition $(iii)$ is satisfied too. Therefore,  under the logical constraint $\no{A}BC=\emptyset$, the family $\{C|B,B|A\}$ p-entails $C|A$, which is another p-valid version of  Weak Transitivity.\\ We observe that  in \cite[Theorem 5]{gilio16} it has been shown that another p-valid  version of Weak Transitivity  is  obtained by adding  the probabilistic constraint
	$P(A|(A\vee B))>0$, that is  
	\begin{equation*}\label{EQ:TRANSEI}
	P(C|B)=1, P(B|A)=1, P(A|(A\vee B))>0\, \Longrightarrow\, P(C|A)=1.
	\end{equation*}
\section{Conclusions}
\label{SECT:CONCLUSIONS}
We generalized the notions of conjunction and disjunction of two conditional events to the case of $n$ conditional events. We  introduced the notion of negation and we showed that De Morgan’s Laws still hold.  We also verified that the associative and commutative properties are satisfied. 
We studied the monotonicity property, by proving that $\C_{n+1}\leq \C_n$ and   $\D_{n+1}\geq \D_n$ for every $n$.
We computed the set of all coherent assessments on  the family $\{\C_n,E_{n+1}|H_{n+1}, \C_{n+1}\}$,
by showing that Fr\'echet-Hoeffding bounds still hold in this case; then,  we  examined the (reverse) probabilistic  inference from $\C_{n+1}$ to the family $\{\C_n,E_{n+1}|H_{n+1}\}$.
Moreover, given a family $\F=\{E_1|H_1,E_2|H_2,E_3|H_3\}$ of three conditional events, 
with $E_1, E_2, E_3, H_1, H_2, H_3$ logically independent,
we determined the set $\Pi$ of all coherent prevision assessments  for the  set of conjunctions $\{\C_{S}: \emptyset \neq S \subseteq \{1,2,3\}\}$. In particular, we verified that the set $\Pi$ is the same in the case where $H_1=H_2=H_3$ and we 
also  considered the relation  between  conjunction and quasi-conjunction. 
By using  conjunction we also characterized  p-consistency and p-entailment; then, we  examined several examples  of p-valid inference rules. We  briefly described a characterization of p-entailment, in the case of two premises, by using a suitable notion of iterated conditioning. Then, after examining some non p-valid inference rules,  we illustrated by an example two methods for constructing  p-valid inference rules. In particular, we applied these methods to Transitivity  by obtaining  p-valid  versions of  the rule (Weak Transitivity). 
Future work could concern the extension of the results of this paper to more complex cases, with  possible applications to the psychology of cognitive reasoning under uncertainty. This work should lead, for instance, to further developments of the results given in \cite{SaPG17,SPOG18}.
\section*{Acknowledgements}
We thank the anonymous referees for their useful criticisms and suggestions.
\appendix
\section{Appendix}
\label{SECT:APPENDIXA}
\begin{proof}\emph{of Theorem \ref{THM:DEMORGAN}.}\\
	We observe that $(ii)$ follows by $(i)$, by replacing $\F$ by $\no{\F}$; indeed,  by $(i)$ it holds that
	$\no{\D(\no{\F})}=\C(\no{\no{\F}})=\C(\F)$. Then, it 
	is enough to proof the assertion $(i)$.
	We will prove the assertion by induction. \\
	Step 1: $n=1, \F=\{E_1|H_1\}$. \\
	We have $\no{\D(\F)}=\no{E_1|H_1}=1-E_1|H_1=\no{E}_1|H_1=\C(\no{\F})$.\\ 
	Thus the assertion holds when $n=1$.\\
	Step 2: $n=2$,  $\F=\{E_1|H_1,E_2|H_2\}$.  \\
	We set 
	\[
	P(E_1|H_1)=x,\, P(E_2|H_2)=y,\, \prev[(E_1|H_1) \vee (E_2|H_2)]=w, \,\prev[(\no{E}_1|H_1) \wedge  (\no{E}_2|H_2)]=t.
	\] 
	We observe that the family $\{
	E_1H_1\vee E_2H_2,
	\no{E}_1H\no{E}_2H_2,
	\no{H}_1\no{E}_2H_2,
	\no{E}_1H_1\no{H}_2,
	\no{H}_1\no{H}_2\}$ is a partition of the sure event $\Omega$. Moreover, by Definitions \ref{CONJUNCTION} and \ref{DEF:DISJUNCTION}  we have
	\begin{equation}\label{}
	\no{\D(\F)}=1-(E_1|H_1)\vee (E_2|H_2) =\left\{\begin{array}{ll}
	0, &\mbox{ if  $E_1H_1\vee E_2H_2$ is true,}\\
	1, &\mbox{ if $\no{E}_1H_1\no{E}_2H_2$  is true,}\\
	1-x, &\mbox{ if $\no{H}_1\no{E}_2H_2$ is true,}\\
	1-y, &\mbox{ if $\no{E}_1H_1\no{H}_2$ is true,}\\
	1-w, &\mbox{ if $\no{H}_1\no{H}_2$ is true}.
	\end{array}
	\right.
	\end{equation}
	and
	\begin{equation}\label{}
	\C(\no{\F})=(\no{E}_1|H_1)\wedge (\no{E}_2|H_2) =
	\left\{\begin{array}{ll}
	0, &\mbox{ if  $E_1H_1\vee E_2H_2$ is true,}\\
	1, &\mbox{ if $\no{E}_1H\no{E}_2H_2$  is true,}\\
	1-x, &\mbox{ if $\no{H}_1\no{E}_2H_2$ is true,}\\
	1-y, &\mbox{ if $\no{E}_1H_1\no{H}_2$ is true,}\\
	t, &\mbox{ if $\no{H}_1\no{H}_2$ is true}.
	\end{array}
	\right.
	\end{equation}
	We observe that $\no{\D(\F)}$ and $
	{\C(\no{\F})}$  coincide when $H_1\vee H_2$ is true. Thus, by Theorem \ref{THM:EQ-CRQ},	$\prev(\no{\D(\F)})=\prev(\C(\no{\F}))$  and hence $1-w=t$. Therefore $\no{\D(\F)}$  still coincides  with $
	\C(\no{\F})$  when $H_1\vee H_2$ is false, so that $\no{\D(\F)}=
	\C(\no{\F})$.	
	\\
	Step 3: $\F=\{E_1|H_1,E_2|H_2,\ldots, E_n|H_n\}$.\\
	(\emph{Inductive Hypothesis}) Let us assume that for any (strict) subset $S\subset\{1,\ldots,n\}$,  
	by defining $\F_S=\{E_i|H_i, i\in S\}$,
	it holds that
	$\no{{\D(\F_S)}}=\C(\no{\F_S})$.
	Now we will prove that $\no{{\D(\F_S)}}=\C(\no{\F_S})$ when $S=\{1,\ldots,n\}$, in which case  $\F_S=\F$.
	By Definition \ref{DEF:DISJUNCTIONn} we have 
	\begin{equation}\label{}
	\begin{array}{lll}
	\no{\D(\F)}=\sum_{h=0}^{m} w_{h}C_h,
	&\mbox{where} &
	w_{h}=
	\left\{
	\begin{array}{llll}
	0, &\mbox{ if } S_h'\neq \emptyset,\\
	1, &\mbox{ if } S_h''=\{1,2,\ldots,n\}, \\
	1-y_{S'''_h}, &\mbox{ if } S_h'=\emptyset \mbox{ and }  S_h''' \neq\emptyset. 
	\end{array}
	\right.
	\end{array}
	\end{equation}
	We continue to use the subsets  $S_h', S_h''$,  $S_h'''$  as defined in formula (\ref{EQ:ESSE}) also with the family $\no{\F}$; moreover we 
	set 	 $t_{S}=\prev[\bigwedge_{i\in S} (\no{E}_i|H_i)]=\prev[{\C(\no{\F}_{S})}]$. Based on Definition \ref{DEF:CONJUNCTIONn},
	we have
	\begin{equation}\label{}
	\begin{array}{lll}
	\C(\no{\F})=\sum_{h=0}^{m} z_{h}C_h,
	&\mbox{where} &
	z_{h}=
	\left\{
	\begin{array}{llll}
	0, &\mbox{ if } S_h'\neq \emptyset,\\
	1, &\mbox{ if } S_h''=\{1,2,\ldots,n\}, \\
	t_{S'''_h}, &\mbox{ if } S_h'=\emptyset \mbox{ and }  S_h''' \neq\emptyset.
	\end{array}
	\right.
	\end{array}
	\end{equation}
	Then, 	$\no{\D(\F)}-\C(\no{\F})=\sum_{h=0}^{m} (w_{h}-z_h)C_h$, where
	\begin{equation}\label{}
	w_{h}-z_h=
	\left\{
	\begin{array}{llll}
	0, &\mbox{ if } S_h'\neq \emptyset,\\
	0, &\mbox{ if } S_h''=\{1,2,\ldots,n\}, \\
	1-y_{S'''_h}-t_{S'''_h}, &\mbox{ if } S_h'=\emptyset \mbox{ and }  S_h''' \neq\emptyset. 
	\end{array}
	\right.
	\end{equation}
	By the inductive hypothesis,  it holds that
	$1-y_{S'''_h}=\prev[\no{{\D(\F_{{S'''_h}})}}]=\prev[\C(\no{\F}_{{S'''_h}})]=t_{S'''_h}$ 
	for $h=1,\ldots,m$, because  $S'''_h\subset\{1,2,\ldots,n\}$. 
	Then, 	$\no{\D(\F)}-\C(\no{\F})=\sum_{h=0}^{m} (w_{h}-z_h)C_h$, where
	\begin{equation}\label{}
	w_{h}-z_h=
	\left\{
	\begin{array}{llll}
	0, & h=1,\ldots,m,\\
	1-y_{S'''_0}-t_{S'''_0}, & h=0.
	\end{array}
	\right.
	\end{equation}	
	By recalling that $S'''_0=\{1,2,\ldots,n\}$,  $\no{\D(\F)}$ and $
	{\C(\no{\F})}$  coincide when $H_1\vee H_2 \vee \cdots \vee H_n$ is true. Thus, by Theorem \ref{THM:EQ-CRQ}, 	$\prev[\no{\D(\F)}]=\prev[\C(\no{\F}))]$, that is $1-y_{S'''_0}=t_{S'''_0}$. Therefore $\no{\D(\F)}$ still coincides with $
	\C(\no{\F})$  when $H_1\vee H_2 \vee \cdots \vee H_n$ is false, so that $\no{\D(\F)}=
	\C(\no{\F})$.
	\qed
\end{proof}
\begin{proof}\emph{of Theorem \ref{THM:INEQ-CRQ}.}\\
	(i) Assume that, for every $(\mu,\nu)\in \Pi$,  the values of $X|H$ and $Y|K$ associated with the  constituent $C_h$ are such that $X|H\leq Y|K$, for each $C_h$ contained in  $H\vee K$; then  for each given coherent assessment $(\mu,\nu)$,  by choosing $s_1=1, s_2=-1$ in the random gain, we have
	\[
	G=H(X-\mu)-K(Y-\nu) = (X|H - \mu) -(Y|K - \nu) = (X|H - Y|K) + (\nu - \mu) \,.
	\]
	Then, by the hypothesis, $\G_{H\vee K}\leq (\nu - \mu)$ and by coherence $0=\pr(\G_{H\vee K}) \leq \nu - \mu\ $. Then   $\mu\leq \nu$, $ \forall (\mu,\nu) \in \Pi$.\\
	
	(ii) By hypothesis, it holds that $(XH+\mu H^c)(H\vee K)\leq (YK+\nu K^c)(H\vee K)$; moreover, from condition (i), $\mu\leq\nu$ for every $(\mu,\nu)\in \Pi$;  then
	\[
	\begin{array}{l}
	X|H = XH+\mu H^c=(XH+\mu H^c)(H\vee K)+(XH+\mu H^c)H^c K^c=\\
	(XH+\mu H^c)(H\vee K)+\mu H^c K^c\leq 	(YK+\nu K^c)(H\vee K)+\nu H^c K^c=\\
	(YK+\nu K^c)(H\vee K)+(YK+\nu K^c)H^c K^c=
	YK + \nu K^c  = Y|K \,.
	\end{array}
	\]
	Vice versa, $X|H\leq Y|K$ trivially implies $X|H\leq Y|K$  when $H\vee K$ is true.
\qed
\end{proof}
\begin{proof}\emph{of Theorem \ref{THM:MONOTONY}.}\\
	We distinguish three cases: $(a)$ the value of $\C_n$ is $0$, with some $E_i|H_i$ false, $i\leq n$; $(b)$ the value of  $\C_n$ is $1$, with  $E_i|H_i$ true, $i=1,\ldots,n$; $(c)$ the value of $\C_n$ is $\prev[\bigwedge_{i\in S} (E_i|H_i)]=\prev(\C_S)=x_{S}$, for some subset $S\subseteq \{1,2,\ldots,n\}$. \\
	Case $(a)$.  It holds that  $\C_{n+1}=0=\C_n$.  \\
	Case  $(b)$.  The value of $\C_{n+1}$ is $1$, or $0$, or $x_{n+1}$, according to whether $E_{n+1}|H_{n+1}$ is true, or  false, or void; thus $\C_{n+1}\leq \C_n$. \\
	Case $(c)$.
	We distinguish three cases: $(i)$ $E_{n+1}|H_{n+1}$ is true; $(ii)$ $E_{n+1}|H_{n+1}$ is false;
	$(iii)$ $E_{n+1}|H_{n+1}$ is  void.
	In the case $(i)$  the value of $\C_{n+1}$ is $x_{S}$, thus $\C_{n+1}=\C_n$.
	In the case $(ii)$  the value of $\C_{n+1}$ is $0$, thus $\C_{n+1}\leq \C_n$.
	In the case $(iii)$  the value of $\C_{n+1}$ is $x_{S\cup \{n+1\}}=\prev[\bigwedge_{i\in S\cup \{n+1\}}(E_i|H_i)]$;  then, in order  to prove that $\C_{n+1}\leq \C_n$, we need to prove that
	$x_{S\cup \{n+1\}}\leq x_S$. We proceed by induction on the cardinality of $S$, denoted by $s$.
	Let be $s=1$, with $\C_S=E_i|H_i$, for some $i\in\{1,\ldots,n\}$.
	We note that	$x_S=\prev(E_i|H_i)=x_i$, $x_{S\cup \{n+1\}}=\prev((E_i|H_i)\wedge (E_{n+1}|H_{n+1}))=x_{\{i,n+1\}}$ and
	by Theorem \ref{THM:FRECHET} it holds that $x_{S\cup \{n+1\}}=x_{\{i,n+1\}}\leq x_i=x_S$.
	Now, let be $s\geq 2$ and  $x_{S\cup \{n+1\}}\leq x_S$  for every $s< n$, so that,   based on Definition~\ref{DEF:CONJUNCTIONn},  $\C_{n+1}\leq \C_n$ when $S$ is a strict subset of $\{1,2,\ldots,n\}$.
	If $S=\{1,2,\ldots,n\}$, as $E_i|H_i$ is void for all $i=1,\ldots,n+1$, it holds that $\C_{n}=\prev(\C_{n})=x_{\{1,\ldots,n\}}$ and $\C_{n+1}=x_{\{1,\ldots,n+1\}}=\prev(\C_{n+1})$ and,
	in order to prove that $\C_{n+1}\leq \C_n$, it remains to prove that   $\prev(\C_{n+1})\leq \prev(\C_{n})$.
	By applying Theorem~\ref{THM:INEQ-CRQ}, with $X|H=\C_{n+1}=Z_{n+1}|(H_1\vee \cdots \vee H_{n+1})$ and $Y|K=\C_{n}=Z_{n}|(H_1\vee \cdots \vee H_{n})$, as	$\C_{n+1}\leq \C_{n}$	when $H_1\vee \cdots \vee H_{n+1}$ is true (i.e., $s<n$ ),  it follows that $\prev(\C_{n+1})\leq \prev(\C_{n})$; therefore $\C_{n+1}\leq \C_n$.	
	\qed
\end{proof}
\begin{proof}\emph{of Theorem \ref{THM:INTCn}.}\\
	Case $(i)$.
	We proceed by induction.
	The property is satisfied for $n=1$; indeed, if $\C_1=E_1|H_1\in \{1,0,x_1\}$, where $x_1=\prev(E_1|H_1) \in[0,1]$, then $\C_1\in[0,1]$.
	Let us assume  that the property holds for $k<n$, that is  $\C_k\in [0,1]$, for every $k<n$.
	Based on Definition~\ref{DEF:CONJUNCTIONn} we distinguish three cases: $(a)$ the value of $\C_n$ is $0$; $(b)$ the value of  $\C_n$ is $1$; $(c)$ the value of $\C_n$ is $\prev[\bigwedge_{i\in S} (E_i|H_i)]=x_{S}$, for some subset $S\subseteq \{1,2,\ldots,n\}$.
	In the cases $(a)$ and $(b)$,  $\C_n\in [0,1]$. In the case $(c)$, if $S=\{i_1,\ldots,i_k\}\subset \{1,2,\ldots,n\}$, then
	$\C_n \in [0,1]$, because $x_S=\prev(\bigwedge_{j=1}^k (E_{i_j}|H_{i_j}))$  is a possible  value of $\C_k=\bigwedge_{j=1}^k (E_{i_j}|H_{i_j})$, with  $k<n$. Finally, if $S=\{1,2,\ldots,n\}$  (that is the conditioning events $H_1,\ldots, H_n$ are all false), then  $\C_n=\prev(\C_n)$  and $\prev(\C_n)\in [0,1]$ because the values of  $\C_n$ restricted to $H_1\vee \cdots \vee H_n$ all belong to $[0,1]$. Therefore $\C_n\in[0,1]$. By a similar  reasoning, based on Definition~\ref{DEF:DISJUNCTIONn} we can prove that $\D_n\in[0,1]$.
	\qed
\end{proof}
\begin{proof}\emph{of Theorem \ref{THM:FRECHETCn}}.\\
	Let $C_0,\ldots, C_m$, with $m=3^{n}-1$ be the constituents associated with $\F_{n+1}=\{E_1|H_1,\ldots, E_{n+1}|H_{n+1}\}$, where $C_0=\no{H}_1\cdots \no{H}_{n+1}$.
	With each $C_h$, $h=1,\ldots,m$, we associate the point $Q_h=(q_{h1},q_{h2},q_{h3})$, which represents the value of the random vector $(\C_n,E_{n+1}|H_{n+1},\C_{n+1})$ when $C_h$ is  true, where
	$q_{h1}$ is the value of $\C_n$, 
	$q_{h2}$ is the value of $E_{n+1}|H_{n+1}$, and 
	$q_{h3}$ is the value of $\C_{n+1}$. 
	With $C_0$ it is associated the point $Q_0=(\mu_n,x_{n+1},\mu_{n+1})=\M$.
	We observe that the set of points $\{Q_h, h=1,\ldots,m\}$ contains in particular 
	the points 
	\[
	Q_1=(1,1,1) \,,\; Q_2=(1,0,0) \,,\; Q_3=(0,1,0) \,,\; Q_4=(0,0,0)\, ,
	\]
	which are respectively associated with the following constituents or logical disjunction of constituents
	\[ 
	E_{1}H_{1} \cdots E_{n}H_{n} E_{n+1}H_{n+1} \,,\;
	E_{1}H_{1} \cdots E_{n}H_{n} \no{E}_{n+1}H_{n+1} \,,\;
	\]
	\[
	(\no{E}_{1}H_{1} \vee  \cdots \vee \no{E}_{n}H_{n}) \wedge E_{n+1}H_{n+1} \,,\;
	(\no{E}_{1}H_{1} \vee  \cdots \vee \no{E}_{n}H_{n}) \wedge \no{E}_{n+1}H_{n+1} \,.
	\]
	Based on Remark \ref{REM:FRECH}, we need to prove that the set of coherent assessments  $\Pi$ on $\{\C_n, E_{n+1}|H_{n+1}, \C_{n+1}\}$ coincides with the convex hull $\I$ of $Q_1,Q_2,Q_3,Q_4$. We recall that coherence of $(\mu_n,x_{n+1},\mu_{n+1})$ implies coherence of all the sub-assessments on the associated subfamilies of 
	$\{\C_n, E_{n+1}|H_{n+1}, \C_{n+1}\}$. 
	The coherence of the single assessments $\mu_n$ on $\C_n$, or 
	$x_{n+1}$ on $E_{n+1}|H_{n+1}$, or $\mu_{n+1}$ on $\C_{n+1}$, simply amounts to conditions 
	\[
	\mu_n \in [0,1] \,,\; x_{n+1} \in [0,1] \,,\;\mu_{n+1} \in [0,1] \,,
	\]
	respectively. Then, by the  hypothesis of logical independence, the sub-assessment $(\mu_n,x_{n+1})$ is coherent, for every $(\mu_n,x_{n+1}) \in [0,1]^2$. By Remark~\ref{REM:MONOTONY}, the coherence of the sub-assessments  $(\mu_n,\mu_{n+1})$ and $(x_{n+1},\mu_{n+1})$  amounts to the conditions $0\leq \mu_{n+1}\leq\mu_{n} \leq 1$ and $0\leq \mu_{n+1}\leq  x_{n+1}\leq 1$. Finally, assuming that the above conditions are satisfied, 
	to prove coherence of   $(\mu_n,x_{n+1},\mu_{n+1})$, by  Theorem \ref{SYSTEM-SOLV}, it is enough to show that the point $(\mu_n,x_{n+1},\mu_{n+1})$ belongs to the convex hull of the points $Q_1,\ldots,Q_m$.
	Moreover, in order $\M$ belongs to the convex hull of $Q_1,\ldots,Q_m$ 
	the following system $(\Sigma)$ must  solvable 
	\begin{equation} \label{EQ:SIGMACONJ}
	\hspace{1cm}\M=\sum_{h=1}^m\lambda_h Q_h,\;\;\sum_{h=1}^m\lambda_h=1,\;\; \lambda_h\geq 0, \;\forall h.
	\end{equation}
	We  show that  the convex hull of the points $Q_1,\ldots,Q_m$ coincides with the convex hull $\I$ of the points  $Q_1,Q_2,Q_3,Q_4$, described in Remark \ref{REM:FRECH}, because all the other points $Q_5,\ldots,Q_m$, are linear convex combinations of $Q_1,Q_2,Q_3,Q_4$, that is  $Q_h\in\I$ for each $h=5,\ldots,m$.
	
	We  examine the following different cases which depend on the logical value of $E_{n+1}|H_{n+1}$: $a)$
	$E_{n+1}|H_{n+1}$ is true; $b)$
	$E_{n+1}|H_{n+1}$ is false; $c)$
	$E_{n+1}|H_{n+1}$ is void.\\
	$a)$ In this case
	\[
	Q_h=(q_{h1},1,q_{h1})=
	q_{h1}(1,1,1)+(1-q_{h1})(0,1,0)=q_{h1} Q_1+(1-q_{h1})Q_3.
	\]
	$b)$ In this case
	\[Q_h=(q_{h1},0,0)=
	q_{h1}(1,0,0)+(1-q_{h1})(0,0,0)=q_{h1} Q_2+(1-q_{h1})Q_4.
	\]
	$c)$ 
	In this case $Q_h=(q_{h1},x_{n+1},q_{h3})$ and we distinguish the following subcases: $(i)$ 
	$\bigwedge_{i=1}^n E_iH_i$ true, so that $Q_{h}=(1,x_{n+1},x_{n+1})$;
	$(ii)$ 
	$\bigvee_{i=1}^n \no{E}_iH_i$ true,
	so that $Q_{h}=(0,x_{n+1},0)$;
	$(iii)$ 
	$E_i|H_i$  void, for every $i \in S$ and  $E_i|H_i$ true for every $i \in \{1,2,\ldots,n\}\setminus S$,. 
	for some $\emptyset \neq S \subset \{1,2,\ldots,n\}$, so that $Q_{h}=(x_S,x_{n+1},x_{S\cup\{n+1\}})$.
	In  subcase $(i)$ it holds that
	\[
	Q_h=(1,x_{n+1},x_{n+1})=
	x_{n+1}(1,1,1)+(1-x_{n+1})(1,0,0)=
	x_{n+1} Q_1+(1-x_{n+1})Q_2.
	\]
	In subcase $(ii)$ it holds that
	\[
	Q_h=(0,x_{n+1},0)=
	x_{n+1}(0,1,0)+(1-x_{n+1})(0,0,0)=
	x_{n+1} Q_3+(1-x_{n+1})Q_4.
	\]
	In subcase $(iii)$, it can be verified by a finite iterative procedure that the point 
	$Q_h=(x_{S},x_{n+1},x_{S\cup\{n+1\}})\in \I$.
	We examine the different cases on the cardinality $s$ of $S$. We recall that $\bigwedge_{i\in S}(E_i|H_i)$ is denoted by $\C_S$.
	\\
	\emph{Step 1}. $s=1$. Without loss of generality we assume $S=\{1\}$,
	so that $Q_h=(x_S,x_{n+1},x_{S\cup\{n+1\}})=(x_1,x_{n+1},x_{\{1,n+1\}})$, where
	$x_1=P(E_1|H_1)$, $x_{\{1,n+1\}}=\prev[(E_1|H_1)\wedge(E_{n+1}|H_{n+1})]$. 
	By Theorem~\ref{THM:FRECHET} it holds that 
	$\max\{x_{S}+x_{n+1}-1,0\} \leq x_{S\cup\{n+1\}} \leq \min\{x_{S},x_{n+1}\}$, with 
	$(x_{S},x_{n+1}) \in [0,1]^2$. In other words,  $Q_h=(x_S,x_{n+1},x_{S\cup\{n+1\}})\in \I$.  The reasoning is the same for $S=\{i\}$, $i=2,\ldots,n$.\\
	\emph{Step 2}. $s=2$. Without loss of generality we assume $S=\{1,2\}$, so that $x_{S}=\prev[(E_1|H_1)\wedge(E_{2}|H_{2})]$,
	$x_{S\cup\{n+1\}}=\prev[\C_{S\cup\{n+1\}}]=\prev[(E_1|H_1)\wedge(E_2|H_2)\wedge(E_{n+1}|H_{n+1})]$. 
	We denote by $C_0^*,C_1^*, \ldots,C_{m^*}^*$, the constituents associated with $\{E_{i}|H_{i}, i \in S\cup\{n+1\}\}$, where $C_0^*=\bigwedge_{i\in S\cup\{n+1\}}\no{H}_i$. Moreover, with $C_h^*$, $h=0,1,	\ldots,m^*,$ we associate the point 
	$Q_h^*=(q_{h1}^*,q_{h2}^*,q_{h3}^*)$ which represents the value of the random vector  $\{\C_{S},E_{n+1}|H_{n+1},\C_{S\cup\{n+1\}}\}$ when $C_h^*$ is true. We observe that 
	$Q_0^*=(x_{S},x_{n+1},x_{S\cup\{n+1\}})$ and that $Q_1,Q_2,Q_3,Q_4$ still belongs to the set of points $\{Q_h^*, h=1,\ldots,m^*\}$.  In order that the assessment $(x_{S},x_{n+1},x_{S\cup\{n+1\}})$ on  $\{\C_{S},E_{n+1}|H_{n+1},\C_{S\cup\{n+1\}}\}$ be 
	coherent, the point $Q_0^*=(x_{S},x_{n+1},x_{S\cup\{n+1\}})$ must belong to the convex hull  of points $Q_1^*, Q_2^*, \ldots, Q_m^*$. 
	We show that for each  point $Q_h^*\neq Q_i$, $i=1,2,3,4$, it holds that $Q_h^*\in \I$.
	By repeating the previous reasoning we only need to analyze the subcase $(iii)$ of case $c)$.
	We have to show that, for every  nonempty subset $S' \subset S$, the point $Q_h^*=(x_{S'},x_{n+1},x_{S'\cup\{n+1\}})$ belongs to the 
	convex hull $\I$ of $Q_1,\ldots,Q_4$. As $S=\{1,2\}$, it holds that $S'=\{1\}$, or $S'=\{2\}$, so that $Q_h^*=(x_{S'},x_{n+1},x_{S'\cup\{n+1\}})=(x_1,x_{n+1},x_{\{1,n+1\}})$, or $Q_h^*=(x_2,x_{n+1},x_{\{2,n+1\}})$. By
	\emph{Step 1}, in both cases $Q_h^* \in \I$. Thus $Q_h=(x_{S},x_{n+1},x_{S\cup\{n+1\}}) \in \I$.
	In other words, $\max\{x_{S}+x_{n+1}-1,0\} \leq x_{S\cup\{n+1\}} \leq \min\{x_{S},x_{n+1}\}$, with 
	$(x_{S},x_{n+1}) \in [0,1]^2$.
	The reasoning is the same for every $S=\{i,j\} \subset \{1,2,\ldots,n\}$.
\[	...............................................................................................................  \]
	\emph{Step $k+1$}. $s=k+1$, $2<k+1<n$. 
	By induction, assume that $(x_{S'},x_{n+1},x_{S'\cup\{n+1\}}) \in \I$ for every $S'=\{i_1,i_2,\ldots,i_k\} \subset \{1,2,\ldots,n\}$. Then, by  the previous reasoning, it follows that $Q_h=(x_{S},x_{n+1},x_{S\cup\{n+1\}}) \in \I$ for every $S=\{i_1,i_2,\ldots,i_{k+1}\}$. 
	In other words, $\max\{x_{S}+x_{n+1}-1,0\} \leq x_{S\cup\{n+1\}} \leq \min\{x_{S},x_{n+1}\}$, with 
	$(x_{S},x_{n+1}) \in [0,1]^2$, for every $S=\{i_1,i_2,\ldots,i_{k+1}\}$.
	
	Thus, by this iterative procedure,  also in the subcase $(iii)$ of case $c)$ it holds that $Q_h\in \I$.
	Then,  $Q_h\in \I$, $h=5,\ldots,m$.
	Finally, the condition (\ref{EQ:SIGMACONJ}) is equivalent to $\M \in \I$, 
	so that 
	the assessment $\M$ is coherent if and only if 
	\[
	(\mu_{n},x_{n+1}) \in [0,1]^2,\;\; \max\{\mu_{n}+x_{n+1}-1,0\} \leq \mu_{n+1} \leq \min\{\mu_{n},x_{n+1}\}.
	\]
\qed	
\end{proof}
\begin{proof}\emph{of Theorem \ref{THM:PIFOR3}.}\\
The computation of the set $\Pi$ is based on Section \ref{Coherence}.  The constituents $C_h$'s and  the  points $Q_h$'s associated with $(\F,\M)$
 are illustrated in Table \ref{TAB:TABLE}. 
 \begin{table}[!ht]
 	\caption{Constituents $C_h$'s and corresponding points $Q_h$'s associated with  $(\F,\M)$,  where    $\mathcal{M}=(x_1,x_2,x_3,x_{12},x_{13},x_{23},x_{123})$ is a prevision assessment on
 		$\F=\{E_1|H_1,E_2|H_2,E_3|H_3,
 		(E_1|H_1)\wedge (E_2|H_2), (E_1|H_1)\wedge (E_3|H_3),
 		(E_2|H_2)\wedge (E_3|H_3), (E_1|H_1)\wedge (E_2|H_2)\wedge (E_3|H_3)\}$.	}
 	\label{TAB:TABLE}
 	\renewcommand*{\arraystretch}{.8}
 	\centering
 	\begin{tabular}{|L|L|LLLLLLL|L|}
 		\hline
 		& C_h                                   &     &     &     &    Q_h & & && \\
 		\hline
 		C_1    & E_1H_1       E_2H_2       E_3H_3      &   1 &   1 &   1 &      1 &      1 &      1 &       1 &	Q_1   \\
 		C_2    & E_1H_1       E_2H_2       \no{E}_3H_3 &   1 &   1 &   0 &      1 &      0 &      0 &       0 &	Q_2   \\
 		C_3    & E_1H_1       E_2H_2       \no{H}_3    &   1 &   1 & x_3 &      1 &    x_3 &    x_3 &     x_3 &	Q_3   \\
 		C_4    & E_1H_1       \no{E}_2H_2  E_3H_3      &   1 &   0 &   1 &      0 &      1 &      0 &       0 &	Q_4   \\
 		C_5    & E_1H_1       \no{E}_2H_2  \no{E}_3H_3 &   1 &   0 &   0 &      0 &      0 &      0 &       0 &	Q_5   \\
 		C_6    & E_1H_1       \no{E}_2H_2  \no{H}_3    &   1 &   0 & x_3 &      0 &    x_3 &      0 &       0 &	Q_6   \\
 		C_7    & E_1H_1       \no{H}_2     E_3H_3      &   1 & x_2 &   1 &    x_2 &      1 &    x_2 &     x_2 &	Q_7   \\
 		C_8    & E_1H_1       \no{H}_2     \no{E}_3H_3 &   1 & x_2 &   0 &    x_2 &      0 &      0 &       0 &	Q_8   \\
 		C_9    & E_1H_1       \no{H}_2     \no{H}_3    &   1 & x_2 & x_3 &    x_2 &    x_3 & x_{23} &  x_{23} &	Q_9   \\
 		C_{10} & \no{E}_1H_1  E_2H_2       E_3H_3      &   0 &   1 &   1 &      0 &      0 &      1 &       0 &	Q_{10}\\
 		C_{11} & \no{E}_1H_1  E_2H_2       \no{E}_3H_3 &   0 &   1 &   0 &      0 &      0 &      0 &       0 &	Q_{11}\\
 		C_{12} & \no{E}_1H_1  E_2H_2       \no{H}_3    &   0 &   1 & x_3 &      0 &      0 &    x_3 &       0 &	Q_{12}\\
 		C_{13} & \no{E}_1H_1  \no{E}_2H_2  E_3H_3      &   0 &   0 &   1 &      0 &      0 &      0 &       0 &	Q_{13}\\
 		C_{14} & \no{E}_1H_1  \no{E}_2H_2  \no{E}_3H_3 &   0 &   0 &   0 &      0 &      0 &      0 &       0 &	Q_{14}\\
 		C_{15} & \no{E}_1H_1  \no{E}_2H_2  \no{H}_3    &   0 &   0 & x_3 &      0 &      0 &      0 &       0 &	Q_{15}\\
 		C_{16} & \no{E}_1H_1  \no{H}_2     E_3H_3      &   0 & x_2 &   1 &      0 &      0 &    x_2 &       0 &	Q_{16}\\
 		C_{17} & \no{E}_1H_1  \no{H}_2     \no{E}_3H_3 &   0 & x_2 &   0 &      0 &      0 &      0 &       0 &	Q_{17}\\
 		C_{18} & \no{E}_1H_1  \no{H}_2     \no{H}_3    &   0 & x_2 & x_3 &      0 &      0 & x_{23} &       0 &	Q_{18}\\
 		C_{19} & \no{H}_1     E_2H_2       E_3H_3      & x_1 &   1 &   1 &    x_1 &    x_1 &      1 &     x_1 &	Q_{19}\\
 		C_{20} & \no{H}_1     E_2H_2       \no{E}_3H_3 & x_1 &   1 &   0 &    x_1 &      0 &      0 &       0 &	Q_{20}\\
 		C_{21} & \no{H}_1     E_2H_2       \no{H}_3    & x_1 &   1 & x_3 &    x_1 & x_{13} &    x_3 &  x_{13} &	Q_{21}\\
 		C_{22} & \no{H}_1     \no{E}_2H_2  E_3H_3      & x_1 &   0 &   1 &      0 &    x_1 &      0 &       0 &	Q_{22}\\
 		C_{23} & \no{H}_1     \no{E}_2H_2  \no{E}_3H_3 & x_1 &   0 &   0 &      0 &      0 &      0 &       0 &	Q_{23}\\
 		C_{24} & \no{H}_1     \no{E}_2H_2  \no{H}_3    & x_1 &   0 & x_3 &      0 & x_{13} &      0 &       0 &	Q_{24}\\
 		C_{25} & \no{H}_1     \no{H}_2     E_3H_3      & x_1 & x_2 &   1 & x_{12} &    x_1 &    x_2 &  x_{12} &	Q_{25}\\
 		C_{26} & \no{H}_1     \no{H}_2     \no{E}_3H_3 & x_1 & x_2 &   0 & x_{12} &      0 &      0 &       0 &	Q_{26}\\
 		C_0    & \no{H}_1     \no{H}_2     \no{H}_3    & x_1 & x_2 & x_3 & x_{12} & x_{13} & x_{23} & x_{123} &	Q_0   \\
 		\hline
 	\end{tabular}
 \end{table}
 We recall that $Q_h=(q_{h1},\ldots,q_{h7})$ represents the value associated with $C_h$ of the random vector
$ (E_1|H_1,E_2|H_2,E_3|H_3, (E_1|H_1)\wedge (E_2|H_2),$ $(E_1|H_1)\wedge (E_3|H_3),
 (E_2|H_2)\wedge (E_3|H_3)$, ${(E_1|H_1)\wedge (E_2|H_2)\wedge (E_3|H_3)})$, $h=1,\ldots,26$. With $C_0=\no{H_1}\no{H_2}\no{H_3}$ it is associated $Q_0=\M$. 
 Denoting by $\mathcal{I}$ the convex hull generated by  $Q_1,Q_2, \ldots,Q_{26}$, the coherence of the prevision assessment $\mathcal{M}$ on $\F$ requires that the condition $\M\in \mathcal{I}$ be satisfied; this amounts to the solvability of the following system
\[\begin{array}{l}
(\Sigma) \hspace{1 cm}
\M=\sum_{h=1}^{26} \lambda_hQ_h,\;\;\;
\sum_{h=1}^{26} \lambda_h=1,\;\;\; \lambda_h\geq 0,\,  \; h=1,\ldots,26 \,.
\end{array}
\]
We observe that
\[
\begin{array}{lll}
Q_3=x_3Q_1+(1-x_3)Q_2,\;\;
Q_6=x_3Q_4+ (1-x_3)Q_5,\;\;\\
Q_7=x_2Q_1+ (1-x_2)Q_4,\;\;
Q_8=x_2Q_2+(1-x_2)Q_5,\;\;  \\
Q_9=x_{23}Q_1+(x_2-x_{23})Q_2+(x_3-x_{23})Q_{4}+(x_{23}-x_2-x_3+1)Q_{5},\\
Q_{12}=x_3Q_{10}+(1-x_3)Q_{11},\;\;
Q_{15}=x_3Q_{13}+(1-x_3)Q_{14},\;\; \\
Q_{16}=x_2Q_{10}+(1-x_2)Q_{13}, \;\;
Q_{17}=x_2Q_{11}+(1-x_2)Q_{14},\;\; \\
Q_{18}=x_{23}Q_{10}+(x_2-x_{23})Q_{11}+(x_3-x_{23})Q_{13}+(x_{23}-x_2-x_3+1)Q_{14},\\
Q_{19}=x_1Q_1+(1-x_1)Q_{10},\;\;
Q_{20}=x_1Q_2+(1-x_1)Q_{11},\\
Q_{21}=x_{13} Q_1+(x_1-x_{13})Q_2  +(x_3-x_{13})Q_{10} +(x_{13}-x_1-x_3+1) Q_{11},\;\;\\
Q_{22}=x_1Q_4+(1-x_1)Q_{13},\;\;
Q_{23}=x_1Q_5+(1-x_1)Q_{14},\;\; \\
Q_{24}=x_{13} Q_4+(x_1-x_{13})Q_5  +(x_3-x_{13})Q_{13} +(x_{13}-x_1-x_3+1) Q_{14},\\
Q_{25}=x_{12}Q_1+(x_1-x_{12})Q_4+(x_2-x_{12})Q_{10}+(x_{12}-x_1-x_2+1)Q_{13},\\
Q_{26}=x_{12}Q_2+(x_1-x_{12})Q_5+(x_2-x_{12})Q_{11}+(x_{12}-x_1-x_2+1)Q_{14}.\\
\end{array}
\]
Thus, $\I$ coincides with the convex hull of the points
$Q_1, Q_2, Q_4, Q_5, Q_{10}, Q_{11}, Q_{13}, Q_{14}$.
For the sake of simplicity, we set:
$Q_1'=Q_1, Q_2'=Q_2, Q_3'=Q_4, Q_4'=Q_5,$ \linebreak $Q_5'=Q_{10}, Q_6'=Q_{11}, Q_7'=Q_{13}, Q_8'=Q_{14}$.
Then,
the condition  $\M\in \I$ amounts to the solvability of the following system
\[\begin{array}{l}
(\Sigma') \hspace{1 cm}
\M=\sum_{h=1}^8 \lambda_h'Q_h',\;\;\;
\sum_{h=1}^8 \lambda_h'=1,\;\;\; \lambda_h'\geq 0,\,  h=1,\ldots, 8  \,
\end{array}
\]
that is
\[
(\Sigma')\left\{
\begin{array}{l}
\lambda_1'+\lambda_2'+\lambda_3'+\lambda_4'=x_1,\;\;
\lambda_1'+\lambda_2'+\lambda_5'+\lambda_6'=x_2,\;\;
\lambda_1'+\lambda_3'+\lambda_5'+\lambda_7'=x_3,\\
\lambda_1'+\lambda_2'=x_{12},\;\;
\lambda_1'+\lambda_3'=x_{13},\;\;
\lambda_1'+\lambda_5'=x_{23},\;\;
\lambda_1'=x_{123},\;\;\\
\sum_{h=1}^8\lambda_h'=1,\;\;
\lambda_h' \geq 0,\;\; h=1,2,\ldots,8.\\
\end{array}
\right.
\]
System $(\Sigma')$ can be written as
\[
(\Sigma')\left\{
\begin{array}{ll}
\lambda'_1=x_{123},\;\;
\lambda'_2=x_{12}-x_{123},\;\;
\lambda'_3=x_{13}-x_{123},\;\;
\lambda'_4=x_1-x_{12}-x_{13}+x_{123},\;\; \\
\lambda'_5=x_{23}-x_{123},\;\;
\lambda'_6=x_2-x_{12}-x_{23}+x_{123},\;\;
\lambda'_7=x_3-x_{13}-x_{23}+x_{123},\;\;\\
\lambda'_8=1-x_1-x_2-x_3+x_{12}+x_{13}+x_{23}-x_{123},\;\;
\lambda'_h \geq 0,\;\; h=1,2,\ldots,8.\\
\end{array}
\right.
\]
As it can be verified, by non-negativity of $\lambda'_1,\ldots,\lambda'_8$ it follows that $(\Sigma')$ is solvable (with a unique solution) if and only if
\begin{equation}\label{EQ:SYSTEMPIcorta}
\left\{
\begin{array}{l}
x_{123}\geq \max\{0,x_{12}+x_{13}-x_1,x_{12}+x_{23}-x_2,x_{13}+x_{23}-x_3\},\\
x_{123}\leq \min\{x_{12},x_{13},x_{23},1-x_1-x_2-x_3+x_{12}+x_{13}+x_{23}\},
\end{array}
\right.
\end{equation}
or, in a more explicit way, if and only if the following conditions are satisfied
\begin{equation}\label{EQ:SYSTEMPI}
\left\{
\begin{array}{l}
(x_1,x_2,x_3)\in[0,1]^3,\\
\max\{x_1+x_2-1,x_{13}+x_{23}-x_3,0\}\leq x_{12}\leq \min\{x_1,x_2\},\\
\max\{x_1+x_3-1,x_{12}+x_{23}-x_2,0\}\leq x_{13}\leq \min\{x_1,x_3\},\\
\max\{x_2+x_3-1,x_{12}+x_{13}-x_1,0\}\leq x_{23}\leq \min\{x_2,x_3\},\\
1-x_1-x_2-x_3+x_{12}+x_{13}+x_{23}\geq 0,\\
x_{123}\geq \max\{0,x_{12}+x_{13}-x_1,x_{12}+x_{23}-x_2,x_{13}+x_{23}-x_3\},\\
x_{123}\leq  \min\{x_{12},x_{13},x_{23},1-x_1-x_2-x_3+x_{12}+x_{13}+x_{23}\}.
\end{array}
\right.
\end{equation}
Notice that the conditions in (\ref{EQ:SYSTEMPI}) coincide with that ones in  (\ref{EQ:SYSTEMPISTATEMENT}). 
Moreover, assuming $(\Sigma')$ solvable, with the solution
$(\lambda_1',\ldots,\lambda_8')$, we associate
the vector $(\lambda_1,\lambda_2,\ldots,\lambda_{26})$, with
$\lambda_1=\lambda_1',\;
\lambda_2=\lambda_2', \;
\lambda_4=\lambda_3', \;
\lambda_5=\lambda_4', \;
\lambda_{10}=\lambda_5', \;
\lambda_{11}=\lambda_6', \;
\lambda_{13}=\lambda_7', \;
\lambda_{14}=\lambda_8', \;
\lambda_h=0, h\notin\{1,2,4,5,10,11,13,14\},
$
which is a solution of $(\Sigma)$. Moreover, defining $\mathcal{J}=\{ 1,2,4,5,10,11,13,14\}$, it holds that
$\bigvee_{h\in\mathcal{J}}C_h= H_1\wedge H_2\wedge H_3$.
Therefore, $\sum_{h\in\mathcal{J}}\lambda_h =\sum_{h: C_h \subseteq H_1H_2H_3} \lambda_h  = 1$ and hence
$\sum_{h: C_h \subseteq H_i}\lambda_h=1$, $i=1,2,3$,
$\sum_{h: C_h \subseteq H_i\vee H_j}\lambda_h=1$, $i\neq j$, $\sum_{h: C_h \subseteq H_1\vee H_2\vee H_3}\lambda_h=~1$; thus, by (\ref{EQ:I0}),  $I_0 = \emptyset$.
Then, by Theorem \ref{CNES-PREV-I_0-INT}, the solvability of $(\Sigma)$ is also sufficient for the coherence of $\M$.  Finally,  $\Pi$ is the set of conditional prevision assessments $(x_1,x_2,x_3,x_{12},x_{13},x_{23},x_{123})$ which satisfy the conditions in  (\ref{EQ:SYSTEMPISTATEMENT}).
\qed
\end{proof}

\begin{proof}\emph{of Theorem \ref{THM:PIFOR3bis}.}\\
Notice that, $(E_i|H)\wedge (E_j|H)=(E_iE_j)|H$, for every $\{i,j\}\subset\{1,2,3\}$, and  $(E_1|H)\wedge (E_2|H)\wedge (E_3|H)=(E_1E_2E_3)|H$. 
Then $\F=\{E_1|H,E_2|H,E_3|H$, $(E_1E_2)|H,$ $(E_1E_3)|H,
(E_2E_3)|H$, $(E_1E_2E_3)|H\}$.
The computation of the set $\Pi$ is based on Section \ref{Coherence}.  The constituents $C_h$'s and  the  points $Q_h$'s associated with $(\F,\M)$
are illustrated in Table \ref{TAB:TABLEbis}. 
\begin{table}[!ht]
		\caption{Constituents $C_h$'s and corresponding points $Q_h$'s associated with  $(\F,\M)$,  where    $\mathcal{M}=(x_1,x_2,x_3,x_{12},x_{13},x_{23},x_{123})$ is a prevision assessment on
	$\F=\{E_1|H,E_2|H,E_3|H$, $(E_1E_2)|H,$ $(E_1E_3)|H,
	(E_2E_3)|H$, $(E_1E_2E_3)|H\}$.	}	
		\renewcommand*{\arraystretch}{.8}
		\centering
		\begin{tabular}{|L|L|LLLLLLL|L|}
			\hline
			& C_h                                   &     &     &     &    Q_h & & && \\
			\hline
C_1    & E_1    E_2     E_3H      &   1 &   1 &   1 &      1 &      1 &      1 &       1 &	Q_1   \\
C_2    & E_1    E_2     \no{E}_3H &   1 &   1 &   0 &      1 &      0 &      0 &       0 &	Q_2   \\
C_3    & E_1    \no{E}_2  E_3H     &   1 &   0 &   1 &      0 &      1 &      0 &       0 &	Q_3   \\
C_4    & E_1    \no{E}_2  \no{E}_3H &   1 &   0 &   0 &      0 &      0 &      0 &       0 &	Q_4   \\
C_{5} & \no{E}_1 E_2     E_3H      &   0 &   1 &   1 &      0 &      0 &      1 &       0 &	Q_{5}\\
C_{6} & \no{E}_1 E_2     \no{E}_3H &   0 &   1 &   0 &      0 &      0 &      0 &       0 &	Q_{6}\\
C_{7} & \no{E}_1 \no{E}_2  E_3H     &   0 &   0 &   1 &      0 &      0 &      0 &       0 &	Q_{7}\\
C_{8} & \no{E}_1 \no{E}_2  \no{E}_3H &   0 &   0 &   0 &      0 &      0 &      0 &       0 &	Q_{8}\\
C_0    & \no{H}    & x_1 & x_2 & x_3 & x_{12} & x_{13} & x_{23} & x_{123} &	Q_0   \\
			\hline
		\end{tabular}
		\label{TAB:TABLEbis}
\end{table}
We recall that $Q_h=(q_{h1},\ldots,q_{h7})$ represents the value associated with $C_h$ of the random vector
$ (E_1|H,E_2|H,E_3|H, (E_1E_2)|H,$ $(E_1E_3)|H,
(E_2E_3)|H$, $(E_1E_2E_3)|H)$, $h=1,\ldots,8$. With $C_0=\no{H}$ it is associated $Q_0=\M$. 
Denoting by $\mathcal{I}$ the convex hull generated by  $Q_1,Q_2, \ldots,Q_{8}$, 
as all the conditioning events coincide with $H$ the  assessment $\mathcal{M}$ on $\F$ is coherent if and only if  $\M\in \mathcal{I}$; that is,  if and only if the following system is solvable
\begin{equation}
\label{EQ:SIGMA}
\begin{array}{l}
\hspace{1 cm}
\M=\sum_{h=1}^{8} \lambda_hQ_h,\;\;\;
\sum_{h=1}^{8} \lambda_h=1,\;\;\; \lambda_h\geq 0,\,  \; h=1,\ldots,8 .
\end{array}
\end{equation}
The points $Q_1,Q_2, \ldots,Q_{8}$ coincide with the points 
$Q_1',Q_2', \ldots,Q_{8}'$ in the proof of Theorem \ref{THM:PIFOR3}, respectively. Then,  system  (\ref{EQ:SIGMA})  coincides with 
system $(\Sigma)'$ in  the proof of Theorem \ref{THM:PIFOR3}. 
Therefore, it is solvable if and only if the conditions in 
(\ref{EQ:SYSTEMPISTATEMENT}) are satisfied.
In other words,  the set $\Pi$ of all coherent assessments $\M$ on $\F$ coincides with 
the set of points  $(x_1,x_2,x_3,x_{12},x_{13},x_{23},x_{123})$ which satisfy the conditions in  (\ref{EQ:SYSTEMPISTATEMENT}).
\qed
\end{proof}
\begin{proof}\emph{of Theorem \ref{THM:PENT}.}\\
In order to prove the theorem it is enough to prove the following implications: a) $(i) \Rightarrow (ii)$; b) $(ii) \Rightarrow (iii)$; c) $(iii) \Rightarrow (i)$.\\
a) $(i) \Rightarrow (ii)$. We recall that $\F$  p-entails $E_{n+1}|H_{n+1}$ if and only if
	either $H_{n+1} \subseteq E_{n+1}$, or there exists a nonempty $\F_{\Gamma} \subseteq \mathcal{F}$, where $\Gamma \subseteq \{1,\ldots,n\}$,  such that $QC(\F_{\Gamma})$ implies $E_{n+1}|H_{n+1}$ (see, e.g. \cite[Theorem 6]{GiSa13IJAR}).
Let us first consider the case where $H_{n+1} \subseteq E_{n+1}$. In this case  $P(E_{n+1}|H_{n+1}) = 1$ and $E_{n+1}|H_{n+1} =H_{n+1}+\no{H}_{n+1}=1$.  We have $\C_{n+1}=\C_n \wedge (E_{n+1}|H_{n+1})$, with $E_{n+1}|H_{n+1} =1$.
 We distinguish two cases: $(\alpha)$ $H_{n+1}$ is true; $(\beta)$ $H_{n+1}$ is false.
In case $(\alpha)$, by Definition \ref{DEF:CONJUNCTIONn} and Remark \ref{REM:CONJUNCTIONn}, as $E_{n+1}|H_{n+1}$ is true it follows that the values of $\C_{n+1}$ and of $\C_n$ coincide.
In case $(\beta)$, let  $C_0,\ldots,C_m$ be the constituents associated with $\F$, where $C_0=\no{H}_1\cdots \no{H}_n$. Then, the constituents $C_0',\ldots,C_m'$ associated with $\F\cup\{E_{n+1}|H_{n+1}\}$ and contained in $\no{H}_{n+1}$ are $C_0'=C_0\no{H}_{n+1},\ldots,C_m'=C_m\no{H}_{n+1}$. 
For each constituent $C_h'$, $h=1,\ldots,m$,  by formula (\ref{EQ:CF}) the corresponding value of $\C_n$ is $z_{h}\in\{1,0,x_{S_{h}'''}\}$. We 
denote by $z_h'$ the value of  $\C_{n+1}$ associated with $z_h$ and we recall that
$C_h'\subseteq \no{H}_{n+1}$, $h=0,1,\ldots,m$.   For each index $h$,
 if $z_h=1$, then $z'_h=1$; if $z_h=0$, then $z'_h=0$; if $z_h=x_{S_{h}'''}$, then $z_h'=x_{S_{h}'''\cup\{n+1\}}$.
We set $P(E_{n+1}|H_{n+1})=x_{n+1}$; in our case $x_{n+1}=1$. Moreover, by Theorem \ref{THM:FRECHETCn} 
\[
\max\{x_{S_{h}'''}+x_{n+1}-1,0\}\leq  x_{S_{h}'''\cup\{n+1\}}\leq \min\{x_{S_{h}'''},x_{n+1}\};
\]
therefore $ x_{S_{h}'''\cup\{n+1\}}=x_{S_{h}'''}$. Then, the values of  $\C_{n+1}$ and of $\C_n$ coincide for every  $C'_h$.
Thus, $\C_{n+1}=\C_n$ when $H_{n+1}\subseteq E_{n+1}$.
\\
We consider now the case where  there exists $\F_{\Gamma}\subseteq \mathcal{F}$, $\F_{\Gamma}\neq \emptyset$, such that  $QC(\F_{\Gamma})\subseteq E_{n+1}|H_{n+1}$. First of all we prove that $\C(\F_{\Gamma}\cup\{E_{n+1}|H_{n+1}\} )=\C(\F_{\Gamma})$. For the sake of simplicity, we set  $\C(\F_{\Gamma})=\C_{\Gamma}$ and  $\C(\F_{\Gamma}\cup\{E_{n+1}|H_{n+1}\} )=\C_{\Gamma\cup \{n+1\}}$.

If  the value of $\C_{\Gamma}$ is 1 (because all the conditional events in $\F_{\Gamma}$ are true), 
then $QC(\F_{\Gamma})$ is true and hence $E_{n+1}|H_{n+1}$ is also true; thus $\C_{\Gamma\cup \{n+1\}}=1$, so that $\C_{\Gamma\cup \{n+1\}}=\C_{\Gamma}$.\\
If  the value of $\C_{\Gamma}$ is 0 (because some  conditional event in $\F_{\Gamma}$ is false),  then $\C_{\Gamma\cup \{n+1\}}$ is 0 too, so that $\C_{\Gamma\cup \{n+1\}}=\C_{\Gamma}$.\\
If $\C_{\Gamma}$ is $x_{S}$   for some nonempty subset $\S\subset \Gamma$  (that is, all the conditional events in $\F_{S}$ are void and the other ones in $\F_{\Gamma \setminus S}$ are true), then $QC(\F_{\Gamma})$ is  true and
and hence $E_{n+1}|H_{n+1}$ is also true; thus $\C_{\Gamma\cup \{n+1\}}=x_{S}$, so that $\C_{\Gamma\cup \{n+1\}}=\C_{\Gamma}$.\\
If $\C_{\Gamma}$ is $x_{\Gamma}$   because  all the conditional events in $\F_{\Gamma}$ are void, then $QC(\F_{\Gamma})$ is  void 
and for $E_{n+1}|H_{n+1}$ there are two cases: 1)  $E_{n+1}|H_{n+1}$  true;  2)  $E_{n+1}|H_{n+1}$  void.
In case 1), by also recalling Remark~\ref{REM:CONJUNCTIONn}, it holds that $\C_{\Gamma\cup \{n+1\}}=x_{\Gamma}$ so that $\C_{\Gamma\cup \{n+1\}}=\C_{\Gamma}$.\\
In case 2) it holds that $\C_{\Gamma\cup \{n+1\}}  =x_{\Gamma \cup\{n+1\}}$, where $x_{\Gamma\cup\{n+1\}}=\prev(\C_{\Gamma\cup \{n+1\}})$. 
Now, we observe that the random quantities $\C_{\Gamma}$ and $\C_{\Gamma\cup \{n+1\}}$ 
 coincide conditionally on $\bigvee_{i\in \Gamma\cup\{n+1\}} H_{i}$ being true; then by Theorem \ref{THM:EQ-CRQ} it holds that 
   $\prev(\C_{\Gamma})=\prev(\C_{\Gamma\cup \{n+1\}})$, that is 
$x_{\Gamma}=x_{\Gamma\cup\{n+1\}}$; thus $\C_{\Gamma\cup \{n+1\}}=\C_{\Gamma}$.\\
Finally, denoting by $\Gamma^c$ the set $\{1,\ldots,n\}\setminus \Gamma$,
 by  the associative property of conjunction we obtain
\[
\C_{n+1}=\C_{n}\wedge E_{n+1}|H_{n+1}=\C_{\Gamma^c}\wedge \C_{\Gamma}   \wedge E_{n+1}|H_{n+1}=\C_{\Gamma^c}\wedge \C_{\Gamma}=\C_n.
\]
b) $(ii) \Rightarrow (iii)$. By monotonicity property of conjunction it holds that $\C_{n+1}\leq E_{n+1}|H_{n+1}$. Then,
by assuming  $\C_{n}=\C_{n+1}$, it follows $\C_{n}\leq E_{n+1}|H_{n+1}$.
\\
c) $(iii) \Rightarrow (i)$. 
Let us  assume that $\C_{n}\leq E_{n+1}|H_{n+1}$, so that $\prev(\C_{n})\leq P(E_{n+1}|H_{n+1})$.
Moreover, by assuming that $P(E_i|H_i)=1$, $i=1,\ldots, n$, from (\ref{EQ:LUKMIN}) 
it follows $\prev(\C_{n})=1$ and hence $P(E_{n+1}|H_{n+1})=1$, that is $\F$ p-entails $E_{n+1}|H_{n+1}$.
\qed
\end{proof}

\begin{thebibliography}{10}
		\bibitem{adams65}
	E.~W. Adams.
	\newblock The logic of conditionals.
	\newblock {\em Inquiry}, 8:166--197, 1965.
	
	\bibitem{adams75}
	E.~W. Adams.
	\newblock {\em The logic of conditionals}.
	\newblock Reidel, Dordrecht, 1975.
	
	\bibitem{benferhat97}
	S.~Benferhat, D.~Dubois, and H.~Prade.
	\newblock Nonmonotonic reasoning, conditional objects and possibility theory.
	\newblock {\em Artificial Intelligence}, 92:259--276, 1997.
	
	\bibitem{biazzo00}
	V.~Biazzo and A.~Gilio.
	\newblock A generalization of the fundamental theorem of de {F}inetti for
	imprecise conditional probability assessments.
	\newblock {\em International Journal of Approximate Reasoning},
	24(2-3):251--272, 2000.
	
	\bibitem{biazzo05}
	V.~Biazzo, A.~Gilio, T.~Lukasiewicz, and G.~Sanfilippo.
	\newblock Probabilistic logic under coherence: {C}omplexity and algorithms.
	\newblock {\em Annals of Mathematics and Artificial Intelligence},
	45(1-2):35--81, 2005.
	
	\bibitem{BiGS08}
	V.~Biazzo, A.~Gilio, and G.~Sanfilippo.
	\newblock Generalized coherence and connection property of imprecise
	conditional previsions.
	\newblock In {\em Proc. IPMU 2008, Malaga, Spain, June 22 - 27}, pages
	907--914, 2008.
	
	\bibitem{BiGS12}
	V.~Biazzo, A.~Gilio, and G.~Sanfilippo.
	\newblock Coherent conditional previsions and proper scoring rules.
	\newblock In {\em Advances in Computational Intelligence. IPMU 2012}, volume
	300 of {\em CCIS}, pages 146--156. Springer Heidelberg, 2012.
	
	\bibitem{boole_1857}
	G.~Boole.
	\newblock {On the Application of the Theory of Probabilities to the Question of
		the Combination of Testimonies or Judgments}.
	\newblock {\em Transactions of the Royal Society of Edinburgh}, 21(4):597--653,
	1857.
	
	\bibitem{Cala87}
	P.~Calabrese.
	\newblock An algebraic synthesis of the foundations of logic and probability.
	\newblock {\em Information Sciences}, 42(3):187 -- 237, 1987.
	
	\bibitem{CaLS07}
	A.~Capotorti, F.~Lad, and G.~Sanfilippo.
	\newblock Reassessing accuracy rates of median decisions.
	\newblock {\em American Statistician}, 61(2):132--138, 2007.
	
	\bibitem{CiDu12}
	D.~Ciucci and D.~Dubois.
	\newblock {Relationships between Connectives in Three-Valued Logics}.
	\newblock In {\em Advances on Computational Intelligence}, volume 297 of {\em
		CCIS}, pages 633--642. Springer, 2012.
	
	\bibitem{CoSc99}
	G.~Coletti and R.~Scozzafava.
	\newblock Conditioning and inference in intelligent systems.
	\newblock {\em Soft Computing}, 3(3):118--130, 1999.
	
	\bibitem{coletti02}
	G.~Coletti and R.~Scozzafava.
	\newblock {\em Probabilistic logic in a coherent setting}.
	\newblock Kluwer, Dordrecht, 2002.
	
	\bibitem{CoSV13}
	G.~Coletti, R.~Scozzafava, and B.~Vantaggi.
	\newblock Coherent conditional probability, fuzzy inclusion and default rules.
	\newblock In R.R. Yager, A.~M. Abbasov, M.~Z. Reformat, and S.~N. Shahbazova,
	editors, {\em Soft Computing: State of the Art Theory and Novel
		Applications}, pages 193--208. Springer,
	Heidelberg, 2013.
	
	\bibitem{CoSV15}
	G.~Coletti, R.~Scozzafava, and B.~Vantaggi.
	\newblock Possibilistic and probabilistic logic under coherence: Default
	reasoning and {S}ystem {P}.
	\newblock {\em Mathematica Slovaca}, 65(4):863--890, 2015.
	
	\bibitem{defi36}
	B.~de~Finetti.
	\newblock La logique de la probabilit\'{e}.
	\newblock In {\em Actes du Congr\`{e}s International de Philosophie
		Scientifique, Paris, 1935}, pages IV 1--IV 9, 1936.
	
	\bibitem{edgington95}
	D.~Edgington.
	\newblock On conditionals.
	\newblock {\em Mind}, 104:235--329, 1995.
	
	\bibitem{gilio02}
	A.~Gilio.
	\newblock Probabilistic reasoning under coherence in {S}ystem {P}.
	\newblock {\em Annals of Mathematics and Artificial Intelligence}, 34:5--34,
	2002.
	
	\bibitem{gilio12ijar}
	A.~Gilio.
	\newblock Generalizing inference rules in a coherence-based probabilistic
	default reasoning.
	\newblock {\em International Journal of Approximate Reasoning}, 53(3):413--434,
	2012.
	
	\bibitem{GOPS16}
	A.~Gilio, D.~Over, N.~Pfeifer, and G.~Sanfilippo.
	\newblock Centering and compound conditionals under coherence.
	\newblock In {\em Soft Methods for Data Science}, volume 456 of {\em AISC},
	pages 253--260. Springer, 2017.
	
	\bibitem{gilio16}
	A.~Gilio, N.~Pfeifer, and G.~Sanfilippo.
	\newblock Transitivity in coherence-based probability logic.
	\newblock {\em Journal of Applied Logic}, 14:46--64, 2016.
	
	\bibitem{GiPS18wp}
	A.~Gilio, N.~Pfeifer, and G.~Sanfilippo.
	\newblock Probabilistic entailment and iterated conditionals, submitted.
	\newblock Available on arxiv: https://arxiv.org/abs/1804.06187v1.
	
	\bibitem{gilio10}
	A.~Gilio and G.~Sanfilippo.
	\newblock Quasi {C}onjunction and p-entailment in nonmonotonic reasoning.
	\newblock In C.~Borgelt, G.~Gonz{\'a}lez-Rodr{\'i}guez, W.~Trutschnig, M.~A.
	Lubiano, M.~{\'A}. Gil, P.~Grzegorzewski, and O.~Hryniewicz, editors, {\em
		Combining Soft Computing and Statistical Methods in Data Analysis}, volume~77
	of {\em {A}dvances in {I}ntelligent and {S}oft {C}omputing}, pages 321--328.
	Springer-{V}erlag, 2010.
	
	\bibitem{GiSa11a}
	A.~Gilio and G.~Sanfilippo.
	\newblock Coherent conditional probabilities and proper scoring rules.
	\newblock In {\em Proc. of ISIPTA'11}, pages 189--198, Innsbruck, 2011.
	
	\bibitem{GiSa13c}
	A.~Gilio and G.~Sanfilippo.
	\newblock Conditional random quantities and iterated conditioning in the
	setting of coherence.
	\newblock In L.~C. van~der Gaag, editor, {\em ECSQARU 2013}, volume 7958 of
	{\em LNCS}, pages 218--229. Springer, Berlin, Heidelberg, 2013.
	
	\bibitem{GiSa13a}
	A.~Gilio and G.~Sanfilippo.
	\newblock {C}onjunction, disjunction and iterated conditioning of conditional
	events.
	\newblock In {\em Synergies of Soft Computing and Statistics for Intelligent
		Data Analysis}, volume 190 of {\em AISC}, pages 399--407. Springer, Berlin,
	2013.
	
	\bibitem{GiSa13IJAR}
	A.~Gilio and G.~Sanfilippo.
	\newblock Probabilistic entailment in the setting of coherence: {T}he role of
	quasi conjunction and inclusion relation.
	\newblock {\em International Journal of Approximate Reasoning}, 54(4):513--525,
	2013.
	
	\bibitem{gilio13}
	A.~Gilio and G.~Sanfilippo.
	\newblock Quasi conjunction, quasi disjunction, t-norms and t-conorms:
	{P}robabilistic aspects.
	\newblock {\em Information Sciences}, 245:146--167, 2013.
	
	\bibitem{GiSa14}
	A.~Gilio and G.~Sanfilippo.
	\newblock Conditional random quantities and compounds of conditionals.
	\newblock {\em Studia Logica}, 102(4):709--729, 2014.
	
	\bibitem{GiSa17}
	A.~Gilio and G.~Sanfilippo.
	\newblock Conjunction and disjunction among conditional events.
	\newblock In S.~Benferhat, K.~Tabia, and M.~Ali, editors, {\em IEA/AIE 2017,
		Part II}, volume 10351 of {\em LNCS}, pages 85--96. Springer, Cham, 2017.
	
	\bibitem{GoNg88}
	I.~R. Goodman and H.~T. Nguyen.
	\newblock {Conditional Objects and the Modeling of Uncertainties}.
	\newblock In M.~M. Gupta and T.~Yamakawa, editors, {\em Fuzzy Computing}, pages
	119--138. North-Holland, 1988.
	
	\bibitem{GoNW91}
	I.~R. Goodman, Hung~T. Nguyen, and Elbert~A. Walker.
	\newblock {\em Conditional Inference and Logic for Intelligent Systems: A
		Theory of Measure-Free Conditioning}.
	\newblock North-Holland, 1991.
	
	\bibitem{hailperin96}
	T.~Hailperin.
	\newblock {\em Sentential probability logic. {O}rigins, development, current
		status, and technical applications}.
	\newblock Lehigh University Press, Bethlehem, 1996.
	
	\bibitem{Kauf09}
	S.~Kaufmann.
	\newblock Conditionals right and left: Probabilities for the whole family.
	\newblock {\em Journal of Philosophical Logic}, 38:1--53, 2009.
	
	\bibitem{lad96}
	F.~Lad.
	\newblock {\em Operational subjective statistical methods: A mathematical,
		philosophical, and historical introduction}.
	\newblock Wiley, New York, 1996.
	
	\bibitem{LSA12}
	F.~Lad, G.~Sanfilippo, and G.~Agr\'{o}.
	\newblock Completing the logarithmic scoring rule for assessing probability
	distributions.
	\newblock {\em AIP Conf. Proc.}, 1490(1):13--30, 2012.
	
	\bibitem{LSA15}
	F.~Lad, G.~Sanfilippo, and G.~Agr\'o.
	\newblock Extropy: complementary dual of entropy.
	\newblock {\em Statistical Science}, 30(1):40--58, 2015.
	
	\bibitem{McGe89}
	V.~McGee.
	\newblock Conditional probabilities and compounds of conditionals.
	\newblock {\em Philosophical Review}, 98(4):485--541, 1989.
	
	\bibitem{Miln97}
	P.~Milne.
	\newblock {Bruno de Finetti and the Logic of Conditional Events}.
	\newblock {\em British Journal for the Philosophy of Science}, 48(2):195--232,
	1997.
	
	\bibitem{PeVa17}
	D.~Petturiti and B.~Vantaggi.
	\newblock Envelopes of conditional probabilities extending a strategy and a
	prior probability.
	\newblock {\em Int. Journal of Approximate Reasoning}, 81:160 -- 182,
	2017.
	
	\bibitem{SaPG17}
	G.~Sanfilippo, N.~Pfeifer, and A.~Gilio.
	\newblock Generalized probabilistic modus ponens.
	\newblock In A.~Antonucci, L.~Cholvy, and O.~Papini, editors, {\em ECSQARU
		2017}, volume 10369 of {\em LNCS}, pages 480--490. Springer, 2017.
	
	\bibitem{SPOG18}
	G.~Sanfilippo, N.~Pfeifer, D.E. Over, and A.~Gilio.
	\newblock Probabilistic inferences from conjoined to iterated conditionals.
	\newblock {\em International Journal of Approximate Reasoning}, 93(Supplement
	C):103 -- 118, 2018.
	
	\bibitem{Verheij2017}
	Bart Verheij.
	\newblock Proof with and without probabilities.
	\newblock {\em Artificial Intelligence and Law}, 25(1):127--154, 2017.
	
\end{thebibliography}

\end{document}